\documentclass{article}

\usepackage[english]{babel}
\usepackage{mathrsfs}
\usepackage{nicefrac}

\usepackage[
style=alphabetic,
]{biblatex}

\usepackage[a4paper,top=2cm,bottom=2cm,left=3cm,right=3cm,marginparwidth=1.75cm]{geometry}

\usepackage[normalem]{ulem}
\usepackage{enumitem}
\usepackage{comment}
\usepackage{amssymb}
\usepackage{amsmath}
\usepackage{amsfonts}
\usepackage{amsthm}
\usepackage{graphicx}
\usepackage{csquotes}
\usepackage{mathtools}
\usepackage[colorlinks=true, allcolors=blue]{hyperref}
\usepackage[capitalize]{cleveref}

\theoremstyle{plain}
\newtheorem{theorem}{Theorem}
\newtheorem*{theorem*}{Theorem}
\newtheorem{lemma}[theorem]{Lemma}
\newtheorem*{lemma*}{Lemma}
\newtheorem{claim}[theorem]{Claim}
\newtheorem{proposition}[theorem]{Proposition}
\newtheorem{corollary}[theorem]{Corollary}
\newtheorem{definition}[theorem]{Definition}

\theoremstyle{remark}
\newtheorem{remark}[theorem]{Remark}
\newtheorem*{remark*}{Remark}

\DeclarePairedDelimiter\floor{\lfloor}{\rfloor}

\newcommand{\length}[1]{\ell(#1)}
\newcommand{\abs}[1]{|#1|}

\newcommand{\eps}{\varepsilon}
\newcommand{\intersect}{\cap}
\newcommand{\union}{\cup}
\newcommand{\grad}{\nabla}

\newcommand{\reals}{\mathbb{R}}

\newcommand{\vol}{\mathrm{Vol}}
\newcommand{\heat}[3]{p_{#1}(#2,#3)}
\newcommand{\mheat}[4]{p_{#1}^{#4}(#2,#3)}

\newcommand{\thin}[1]{{#1}^{\mathrm{thin}}}
\newcommand{\thick}[1]{{#1}^{\mathrm{thick}}}
\newcommand{\truncated}[1]{\tilde{#1}}
\newcommand{\im}[1]{\mathrm{Im} \left(#1\right)}
\newcommand{\re}[1]{\mathrm{Re} \left(#1\right)}
\newcommand{\inj}[1]{\mathrm{Inj}\left(#1\right)}
\newcommand{\norm}[1]{\|#1\|}
\newcommand{\hyperbolic}{{\mathbb{H}}}
\newcommand{\mdist}[3]{\mathrm{dist}_{#3}\left(#1, #2 \right)}
\newcommand{\dist}[2]{\mathrm{dist}\left(#1, #2 \right)}

\renewcommand{\S}{{S}} 
\newcommand{\mS}{{\tilde{\S}}} 
\newcommand{\curveweight}[1]{w \left(#1 \right)}
\newcommand{\truncinj}[1]{\widehat{\mathrm{Inj}}\left({#1} \right)}
\newcommand{\w}[1]{w\left(#1 \right)}
\newcommand{\aw}{I}
\newcommand{\wdist}[2]{\mathrm{dist}_w\left(#1, #2 \right)}
\newcommand{\wball}[2]{B_w\left(#1, #2 \right)}
\newcommand{\el}[1]{\mathsf{EL}\left( #1 \right)}
\newcommand{\good}[1]{G_{#1}}
\newcommand{\smallball}[2]{A_{#1, #2}}
\newcommand{\tand}{\textrm{ and }}

\usepackage{xcolor}

\title{A sharp lower bound on the small eigenvalues of surfaces}
\author{Renan Gross, Guy Lachman and Asaf Nachmias\footnote{Department of mathematical sciences, Tel Aviv University. \{renang,lachman,asafnach\}@tauex.tau.ac.il}}

\begin{document}

\maketitle

\begin{abstract}
    Let $\S$ be a compact hyperbolic surface of genus $g\geq 2$ and let $\aw(\S) = \frac{1}{\vol(\S)}\int_{\S} \frac{1}{\inj{x}^2 \wedge 1} dx$, where $\inj{x}$ is the injectivity radius at $x$.
    We prove that for any $k\in \{1,\ldots, 2g-3\}$, the $k$-th eigenvalue $\lambda_k$ of the Laplacian satisfies
    \begin{equation*}
      \lambda_k \geq  \frac{c k^2}{\aw(\S) g^2} \, ,
    \end{equation*}    
    where $c>0$ is some universal constant. We use this bound to prove the heat kernel estimate
    \begin{equation*}
        \frac{1}{\vol(\S)} \int_\S \Big| \heat{t}{x}{x} -\frac{1}{\vol(\S)} \Big | ~dx \leq  C \sqrt{ \frac{\aw(\S)}{t}} \qquad \forall t \geq 1 \, ,
    \end{equation*}    
    where $C<\infty$ is some universal constant. These bounds are optimal in the sense that for every $g\geq 2$ there exists a compact hyperbolic surface of genus $g$ satisfying the reverse inequalities with different  constants.
\end{abstract}


\section{Introduction} 
Let $\S$ be a compact hyperbolic surface of genus $g\geq 2$. Due to compactness, the Laplacian has a discrete spectrum $0=\lambda_0 < \lambda_1 \leq \lambda_2 \leq \ldots$ with $\lambda_j \to \infty$ as $j \to \infty$. 
The eigenvalues below $1/4$ are usually called \emph{small}, see \cite[p386]{Huber61} and \cite[Chapter 8]{buser_book}.
A classical construction of Buser \cite{buser_riemann_surfaces_with_eigenvalues_in_0_14} shows that for any $\eps>0$ and any $g\geq 2$ there exists a compact hyperbolic surface $\S$ of genus $g$ with $\lambda_k \leq \eps$ for all $k=1,\ldots, 2g-3$. Otal and Rosas \cite{otal_rosas_eigenvalues_must_be_large} showed that $\lambda_{2g-2} > 1/4$ for any $\S$, hence only the first $2g-3$ eigenvalues can be small (for more refined results, see \cite{mondal_systole_of_closed_hyperbolic_surfaces}). Here we provide a sharp estimate as to how small the small eigenvalues can be in terms of $k$ and $g$ and the quantity 
\begin{equation*}
    \aw(\S) =\frac{1}{\vol(\S)}\int_{\S} \frac{1}{\inj{x}^2\wedge 1} dx \, ,
\end{equation*}
where $\inj{x}$ is the injectivity radius at $x$.

\begin{theorem}\label{thm:main_eigenvalue_lower_bound} There exists a universal constant $c>0$ such that for any compact hyperbolic surface $\S$ of genus $g\geq 2$ and every $k\in \{1, \ldots, 2g-3\}$,   
\begin{equation} \label{eq:main_eigenvalue_lower_bound}
   \lambda_k \geq \frac{c k^2}{\aw(\S) g^2} \, .
\end{equation}
\end{theorem}

We use this estimate to provide a uniform bound on the heat kernel trace on $\S$.

\begin{theorem}\label{thm:main_heat_kernel}For any compact hyperbolic surface $\S$ 
     \begin{equation} \label{eq:heatkernelBound}
         \frac{1}{\vol(\S)} \int_\S \Big| \mheat{t}{x}{x}{\S} -\frac{1}{\vol(\S)} \Big | ~dx \leq C \sqrt{\frac{\aw(\S)}{t}} \qquad \forall t \geq 1 \, . 
     \end{equation}
where $\mheat{t}{x}{y}{\S}$ is the heat kernel of the Laplacian on $\S$ and $C<\infty$ is a universal constant.
\end{theorem}

\begin{remark*}
Throughout this paper we use the term  ``universal constant'' to mean a real number that does not depend on $\S$, $g$, $k$, or any other parameter.    
\end{remark*}

We also show that \cref{thm:main_eigenvalue_lower_bound} and \cref{thm:main_heat_kernel} are optimal in the sense that for any $g\geq 2$ and $I\geq 1$, there exists a compact hyperbolic surface $\S$ of genus $g$ and $\aw(\S)
\geq I$ such that the reversed inequalities of \eqref{eq:main_eigenvalue_lower_bound} and \eqref{eq:heatkernelBound}
 hold with a different constant.

We remark that $\aw(\S)$ has an additional simple geometric interpretation. Let $\gamma_1,\ldots, \gamma_s$ be the set of all simple closed geodesics in $\S$ of length at most $2\sinh^{-1}(1)$ and denote by $\{\ell(\gamma_i)\}_{i=1}^s$ their lengths. Then it is not hard to prove (see \cref{cor:geodesic_integral_correspondence}) that 
        \begin{equation*}
            \aw(\S) \asymp 1 + \frac{1}{\vol(\S)}\sum_{i=1}^s \frac{1}{\ell(\gamma_i)} \, ,
        \end{equation*}
where we write $A \lesssim B$ if there exists a universal constant $C \in(0,\infty)$ such that $A \leq C B$ and $A \asymp B$ if both $A \lesssim B$ and $B \lesssim A$.

\subsection{Related work} A classical paper of Schoen, Wolpert and Yau \cite{schoen_wolpert_yau} provides the bound $\lambda_k \geq \alpha(g) L_k(\S)$, where $\alpha(g)$ is an unknown function of $g$, and $L_k(\S)$ is the minimal possible sum of the lengths of simple closed geodesics in $\S$ that cut $\S$ into $k + 1$ components. Providing an explicit dependence on $g$ in this bound, it is shown in \cite{wu_xue_optimal_lower_bounds_for_first_eigenvalues} that $\lambda_1 \gtrsim L_1(\S)/g^2$ and in \cite{he_wu_second_eigenvalues_of_closed_hyperbolic_surfaces} that $\lambda_2 \gtrsim L_2(\S)/g^2$.  Another general lower bound was obtained in \cite{dubashinskiy2019spectra}, showing that $\lambda_{\lceil \eps g \rceil} \geq c(\inj{\S}) \eps^2$, where  $c(\inj{\S})>0$ is a constant depending only on the injectivity radius of $S$. 

We remark that while both quantities $L_k(\S)$ and $1/\aw(\S)$ measure the amount of bottlenecks in a surface (see \cref{cor:geodesic_integral_correspondence}), they are in general incomparable. Whenever the injectivity radius is uniformly bounded below by a universal constant, \cref{thm:main_eigenvalue_lower_bound} implies that $\lambda_k \gtrsim k^2/g^2$ for $k\in\{1,\ldots,2g-3\}$ improving a result obtained in \cite{WuXue22} regarding such lower bounds in any thick part of the moduli space. 

On-diagonal heat kernel bounds such as \eqref{eq:heatkernelBound} are a well studied topic going back to Nash, Moser and Varopoulos, see  \cite{grigoryan_heat_kernel_upper_bounds, grigoryan_heat_kernel_and_analysis_on_manifolds,Grigoyan1999}. They are classically proved using a suitable isoperimetric inequality and typically yield a uniform estimate in $x$. In this paper, however, one can see that isoperimetric inequalities are irrelevant; for example, one can take $O(1)$ primitive closed geodesic of length $\Theta(1)$, pinch them so that their length is $\Theta(1/g)$ and $\aw(\S)$ will remain of the same order, leaving the bounds of Theorems \ref{thm:main_eigenvalue_lower_bound} and \ref{thm:main_heat_kernel} unaltered. Additionally, one cannot hope for a bound on the on-diagonal heat kernel that is uniform in $x$ without some global isoperimetric inequality. 

\subsubsection{Discrete analogue}\label{sec:discreteAnalogue} 
Laplacian eigenvalues and heat kernel bounds are a well studied topic in the discrete setting as well. Let $G$ be a simple, connected, regular graph with $n$ vertices and consider the \emph{lazy random walk} on it; that is, the walker moves to a uniformly drawn neighbor with probability $1/2$ or stays put otherwise. Let $P$ denote its transition probability matrix. The \emph{graph Laplacian} $\mathcal{L}=I-P$ has eigenvalues $0=\lambda_0 < \lambda_1 \leq \ldots \lambda_{n-1} \leq 1$. The bound 
\begin{equation}\label{eq:graphEigenvalueLowerBound}
    \lambda_k \gtrsim \frac{k^2}{n^2} \, ,
\end{equation}
for all $k\in\{1,\ldots,n\}$ is well known, and follows directly from the estimate (see \cite{CKS87,Coulhon2000,lyons_asymptotic_enumeration_of_spanning_trees, lyons_oveis_gharan_sharp_bounds_via_spectral_embedding}): 
\begin{equation}\label{eq:discreteHeatKernel}
   \frac{1}{n} \Big |\sum_x P^t(x,x) - 1 \Big | \lesssim \frac{1}{\sqrt{t}} \, , 
\end{equation}
for any integer $t\geq 1$ (this is the discrete analogue of \eqref{eq:heatkernelBound}). Indeed, recalling that $\sum_x P^t(x,x) = \sum_{i=0}^{n-1} (1-\lambda_i)^t$ and 
plugging in $t=\lceil \lambda_k^{-1} \rceil$ to the above immediately yields \eqref{eq:graphEigenvalueLowerBound}. In this setup it is also known that the $t^{-1/2}$ upper bound on $|P^t(x,x)-n^{-1}|$ holds uniformly in $x$. Of course we cannot hope for the bound \eqref{eq:graphEigenvalueLowerBound} to be valid for all surfaces, nor can we expect a uniform bound in $x$ for same reason, namely, the possibility of arbitrarily small bottlenecks.

Thus, the more appropriate discrete analogue are finite \emph{weighted} graphs. In this setting, bounds on the spectral measure at a vertex were given by Lyons and Oveis Gharan \cite[][Corollary 4.8]{lyons_oveis_gharan_sharp_bounds_via_spectral_embedding}. Similar bounds were obtained by Lyons and Judge \cite{Judge_Lyons_2019} in the setup of \emph{homogeneous} Riemannian manifolds. We were greatly inspired by the use of the spectral kernel in \cite{lyons_oveis_gharan_sharp_bounds_via_spectral_embedding}; as done there, we too bound the eigenvalues by the Dirichlet energy of a function associated with the spectral kernel, but the bulk of this paper is dedicated to bounding below this energy using a novel geometric argument based on extremal length. 

\subsection{Extremal length} \label{sec:outline}

Extremal length is a geometric method in complex analysis that has had a profound influence on the theory of conformal and quasi-conformal mappings. In this paper we obtain an explicit geometric bound on the extremal length between two  small geodesic discs placed at arbitrary locations in $\S$. 

Given a Borel measurable function $\rho : \S \to [0,\infty)$ and a rectifiable curve $\gamma \subseteq \S$, the $\rho$-length of $\gamma$ is
    \begin{equation*}
        L(\gamma, \rho) = \int_\gamma \rho ~|dz| \, .
\end{equation*}

\begin{definition}[Extremal length] Given a collection $\Gamma$ of rectifiable curves in $\S$ we define the extremal length of $\Gamma$ as 
    \begin{equation*}
        \el{\Gamma} = \sup_\rho  \frac{\inf_{\gamma \in \Gamma} L(\gamma,\rho)^2}{\int_{\S}\rho^2 ~d\mu} \, ,
    \end{equation*}
    where the supremum is taken over all Borel measurable functions $\rho : \S \to [0,\infty)$ and $\mu$ is the Riemannian area measure of $\S$.
\end{definition}

For a general Riemannian manifold $M$, we write $B_M(x,r)$ for the ball in $M$ of radius $r$ around $x \in M$. The injectivity radius of a point $x \in M$, denoted $\inj{x, M}$, is the supremum radius $r$ such that the exponential map $\exp_x(r)$ is a diffeomorphism. We often write $\inj{x}$ when it is clear what $M$ is from context. The injectivity radius of the manifold $M$ itself is $\inj{M} = \inf_{x \in M} \inj{x,M}$. When $\S$ is a Riemannian surface of constant curvature $-1$, i.e., a hyperbolic surface, $\inj{x, \S}$ is the supremum  $r$ so that the ball of radius $r$ around $x$ is isometric to the ball of radius $r$ in the hyperbolic plane $\hyperbolic$.

\begin{definition}[Reciprocal-injectivity weight] \label{def:weight}
    For $x \in \S$, let 
    \begin{equation*}
        \truncinj{x}= \inj{x, \S} \land 1 \, .
    \end{equation*}
    For a rectifiable curve $\gamma \subseteq \S$, let
    \begin{equation*}
        \curveweight{\gamma} = \int_\gamma \truncinj{z}^{-1}\abs{dz} \, .
    \end{equation*}    
    This lets us define the \emph{weighted distance} between $x$ and $y$ as 
    \begin{equation*}
        \wdist{x}{y} = \inf_{\gamma} \curveweight{\gamma} \, ,    
    \end{equation*}
    where the infimum ranges over all rectifiable curves $\gamma$ which connect $x$ to $y$.
\end{definition}

Our main geometric result is the following.

\begin{theorem}\label{thm:main_extremal_length} 
    Let $x,y \in \S$ and $r_x, r_y$ be positive numbers satisfying 
    \begin{eqnarray*}
       d_\S(x,y)\geq r_x+r_y \, ,\quad r_x \leq  \truncinj{x}/2 \, , \quad r_y \leq \truncinj{y}/2 \, . 
    \end{eqnarray*}        
    Let $\Gamma$ be the set of all curves in  $\S \backslash \left( B_\S(x, r_x) \union B_\S(y,r_y) \right)$ from $\partial B_\S(x,r_x)$ to $\partial B_\S(y,r_y)$. Then 
    \begin{equation}\label{eq:main_extremal_length}
        \el{\Gamma} \lesssim \wdist{x}{y} + \log\left(\frac{\truncinj{x}}{r_x} \right) + \log\left(\frac{\truncinj{y}}{r_y}  \right) \, .
    \end{equation}            
\end{theorem}
\begin{remark}\label{rmk:weightedgraph} Let us provide a rough, yet guiding, intuition. It is well known that any compact hyperbolic surface $\S$ has  a pair of pants decomposition in which all simple closed geodesics of length at most $2\sinh^{-1}(1)$ are boundary geodesics where two pairs of pants were glued. This is the content of the so-called collar lemma (\cref{lem:collar_lemma}). If a path $\gamma$ traverses from one side of a narrow collar of width $\ell$ to the other, then the contribution of this part of the curve to its weight $\w{\gamma}$ is proportional to $1/\ell$. Thus, $\w{\gamma}$ is the sum of the reciprocal lengths of the narrow collars through which it passes, up to a multiplicative constant.

We may think of this pair of pants decomposition of $\S$ as defining a $3$-regular graph in which the pants are vertices and edges are formed between two pairs of pants if they are glued at one of their three boundary geodesics. Assign to each edge a weight $c(e)$ equal to the length of the corresponding boundary geodesic. The weighted distance between two vertices $x,y$ in this analogous discrete setting is the minimal sum of reciprocals edge weights $1/c(e)$ over all paths from $x$ to $y$. 

If the weights $\{c(e)\}$  are thought of as  electric conductances assigned to the edges, then by the parallel law together with Rayleigh's monotonicity law, the effective electric resistance between $x$ and $y$ is at most the sum of edge resistances, i.e. $1/c(e)$, over the edges of any path from $x$ to $y$. So the weighted distance between two vertices $x$ and $y$ gives an upper bound on the effective electric resistance (this bound is far from sharp in many cases). Extremal length is the continuous analogue of effective electric resistance and so \cref{thm:main_extremal_length} is a very rough analogue of the aforementioned bound. 
\end{remark}

\subsection*{Organization}
The rest of this paper is organized as follows. In the next section, we provide the necessary geometric background. In \cref{sec:extremal_length_proof} we prove \cref{thm:main_extremal_length} and in 
\cref{sec:lower_bound_proof} we use it to prove \cref{thm:main_eigenvalue_lower_bound}. 
\cref{thm:main_heat_kernel} is proved in \cref{sec:heat_kernel_proof} and optimality of \cref{thm:main_eigenvalue_lower_bound} and \cref{thm:main_heat_kernel} is shown in \cref{sec:sharpness}.

\section{Geometric preliminaries}
A hyperbolic surface is a connected orientable Riemannian surface with constant curvature $-1$. An important geometric characterization of compact hyperbolic surfaces is the so-called collar lemma, which describes the parts of the surface with small injectivity radius. The following is a combination of Theorem 4.1.1 and Theorem 4.1.6 from \cite{buser_book}.
\begin{lemma}[Collar lemma] \label{lem:collar_lemma}
    Let $\gamma_1, \ldots, \gamma_s$ be the set of all simple closed geodesics of length at most $2\sinh^{-1}(1)$ on a compact hyperbolic surface $\S$. Let $W(\gamma_i) =\sinh^{-1}\left(\frac{1}{\sinh\left(\frac{1}{2} \length{\gamma_i}\right)}\right)$ and $C(\gamma_i) = \{x\in \S \mid \mdist{x}{\gamma_i}{\S} \leq W(\gamma_i)\}$. Then 
    \begin{itemize}        
        \item The sets $C(\gamma_i)_{i=1}^s$ are pairwise disjoint.
        \item $\inj{x} \geq \sinh^{-1}(1)$ for all $x \notin \union_i C(\gamma_i)$.
        \item Each $C(\gamma_i)$ is isometric to the cylinder $\left[-W(\gamma_i), W(\gamma_i)\right] \times \left(\reals / [x \mapsto x + 1] \right) $
        with the Riemannian metric $ds^2 = d\rho^2 + \length{\gamma_i}^2\cosh(\rho)^2 d\theta^2$.
        \item If $x \in C(\gamma_i)$ is such that $\inj{x} \leq \sinh^{-1}(1)$ and $d = \mdist{x}{\partial C(\gamma_i)}{\S}$, then
        \begin{equation*}
            \sinh(\inj{x}) = \cosh\left(\frac{1}{2}\length{\gamma_i}\right)\cosh(d) - \sinh(d).
        \end{equation*}
    \end{itemize}
\end{lemma}

The following proposition provides a formula and an estimate of the injectivity radius of points in the collars in terms of their cylindrical coordinates.
\begin{proposition}[Injectivity radius estimate]\label{prop:injradius_estimate}
Let $T$ be an infinite hyperbolic collar whose shortest closed geodesic $\gamma$ has length $\ell$. Denote by $r(\rho,\ell)$ the injectivity radius at a point of distance $\rho$ from $\gamma$. Then
\begin{equation*}
    r(\rho, \ell) = \frac{1}{2}\cosh^{-1}\left(1 + \left(\cosh(\ell) -1 \right)\cosh(\rho)^2 \right) \, .
\end{equation*}
In particular, 
\begin{equation*}
    r(\rho, \ell) \asymp \ell \cosh(\rho) \, 
\end{equation*}
for all $\ell \leq 2\sinh^{-1}(1)$ and all $\rho \leq W(\gamma)$ where $W(\gamma)$ is defined in \cref{lem:collar_lemma}.
\end{proposition}
\begin{proof}
In the upper-half plane model, an infinite collar can be realized as the quotient of $\hyperbolic$ by the relation $x \sim \lambda x$, where $\lambda \in \reals$ is chosen such that the hyperbolic distance satisfies $d_{\hyperbolic}(i,\lambda i) = \ell$. Since $i$ and $\lambda i$ are on the imaginary axis, a simple integration gives $\log(\lambda) = \ell$. The injectivity radius of a  point $x$ is then given by $\frac{1}{2}d_{\hyperbolic}(x,e^\ell x)$. To calculate this quantity, let $x$ be a point in the upper-half plane whose hyperbolic distance from the imaginary axis is $\rho$. Let $r=|x|$, so that $d_{\hyperbolic}(x,ir)=\rho$. The distance between $z,w\in\hyperbolic$ is given by
\begin{equation*}
    \cosh(d_{\hyperbolic}(z,w)) = 1 + \frac{|z-w|^2}{2 \im{z} \im{w}} \, .
\end{equation*} 
This lets us find the imaginary part of $x$. Using the distance formula, we have
\begin{align*}
    \cosh(\rho) -1 &=\cosh(d_{\hyperbolic}(x,ir))-1\\
    &=\frac{|x-ir|^2}{2\im{x}r}=\frac{\re{x}^2+(\im{x}-r)^2}{2\im{x}r}=\frac{2r^2-2\im{x}r}{2\im{x}r}=\frac{r}{\im{x}}-1 \, ,
\end{align*}
and so $\im{x}=r/\cosh(\rho)$. Hence
\begin{equation*}
    \cosh(d_{\hyperbolic}(x,e^\ell x))=1 + \frac{|x|^2\left(e^\ell-1\right)^2}{2e^\ell \im{x}^2} = 1 + \frac{\left(e^\ell-1\right)^2 \cosh(\rho)^2}{2e^\ell},
\end{equation*}
yielding
\begin{equation}\label{eq:distance_in_collar}
    d_{\hyperbolic}(x,e^\ell x)= \cosh^{-1}\left(1 + \left(\cosh(\ell) -1 \right)\cosh(\rho)^2 \right) \, .
\end{equation}
The asymptotic result then follows from the asymptotic expansions $\cosh^{-1}(1+x) = \sqrt{x} + O(x^{3/2})$ and $\cosh(\ell)-1 = \ell^2/2 + O(\ell^4)$, and observing that $W(\gamma)\leq \log(1/\ell) + C$ for some universal constant $C$ as long as $\ell \leq 2$. 
\end{proof}

\begin{corollary}\label{cor:geodesic_integral_correspondence}
        Let $\S$ be a compact hyperbolic surface, and let $\gamma_1, \ldots, \gamma_s$ be the set of all simple closed geodesics of length $\leq 2\sinh^{-1}(1)$ in $\S$. Then 
        \begin{equation*}
            \aw(\S) \asymp 1 + \frac{1}{\vol(\S)}\sum_{i=1}^s \frac{1}{\ell(\gamma_i)} \, .
        \end{equation*}
    \end{corollary}
\begin{proof} Consider a pair of pants decomposition which uses $\gamma_1, \ldots, \gamma_s$ as boundary geodesics. Using the metric $ds^2 = d\rho^2 + \ell(\gamma_i)^2\cosh(\rho)^2d\theta^2$ from \cref{lem:collar_lemma}, the volume of the intersection of a pair of pants with a collar $C(\gamma_i)$ is given by 
\begin{equation*}
\int_0^1 \int_0^{W(\gamma_i)} \ell(\gamma_i) \cosh(\rho)d\theta d\rho = \ell(\gamma_i) \sinh(W(\gamma_i)) = \frac{\ell(\gamma_i)}{\sinh\left(\frac{1}{2}\ell(\gamma_i)\right)} \leq 2 .
\end{equation*}
Each pair of pants has volume $2\pi > 6$ and intersects at most three collars $C(\gamma_i)$, hence
    \begin{equation*}
        \int_{\S \setminus \union_i C(\gamma_i)} \frac{1}{\inj{x}^2\land 1}~dx \asymp \vol(\S) \, .
    \end{equation*}
    For points $x \in \union_i C(\gamma_i)$, the integrand is proportional to $1/\inj{x}^2$. For every collar $C(\gamma_i)$, by \cref{prop:injradius_estimate} we therefore have
    \begin{align*}
        \int_{C(\gamma_i)} \frac{1}{\inj{x}^2\land 1}~dx &\asymp \int_{C(\gamma_i)} \frac{1}{\inj{x}^2}~dx \\
        &\asymp \int_0^{1} \int_0^\infty \frac{1}{\ell(\gamma_i)^2 \cosh(\rho)^2} \ell(\gamma_i)\cosh(\rho) ~d\rho d\theta 
        \asymp \frac{1}{\ell(\gamma_i)} \, ,
    \end{align*}
    and the result follows.
\end{proof}

\begin{lemma}\label{lem:weight_is_subexponential}
    Let $\S$ be a compact hyperbolic surface, and let $x,y\in \S$ with $\mdist{x}{y}{\S} \leq 1$. Then
    \begin{equation*}
        \truncinj{x} \asymp \truncinj{y} \, .
    \end{equation*}
\end{lemma}
\begin{proof}
We appeal to \cref{lem:collar_lemma}. We use the term \emph{thin part} for the set $\union_{i=1}^s C(\gamma_i)$ and \emph{thick part} for its complement. There are four cases for the relative location of $x$ and $y$. Firstly, if $x$ and $y$ are both in the thick part, then both $\truncinj{x}$ and $\truncinj{y}$ are bounded from below by $\sinh^{-1}(1)$, so $\truncinj{x} \asymp \truncinj{y}$. Secondly, if $x$ is in the thick part and $y$ is in the thin part, then $y$ is at distance at most $1$ from $\partial C(\gamma_i)$ for some $i$, and so by the fourth item in \cref{lem:collar_lemma}, with $d\leq 1$, if $\inj{y,S} \leq \sinh^{-1}(1)$, then
    \begin{align*}
        \inj{y,\S} &= \sinh^{-1}\left(\cosh\left(\frac{1}{2}\length{\gamma_i}\right)\cosh(d) - \sinh(d) \right) \\
        &\geq \sinh^{-1}\left(\cosh(d) - \sinh(d) \right) \geq \sinh^{-1}\left(\cosh(1) - \sinh(1) \right) \, .
    \end{align*}
We learn that in this case the injectivity radius at $y$ is bounded from below by a universal constant, and so $\truncinj{y} \asymp \truncinj{x}$ follows. Thirdly, if $x$ and $y$ are both in the thin part but belong to different collars, then both $x$ and $y$ are at distance at most $1$ from the boundary of some collar, and the same calculation as above apply again to show that both injectivity radii are bounded below. Lastly, suppose $x$ and $y$ are both in the thin part and belong to the same collar $C(\ell(\gamma))$. Writing $x = (\rho_1,\theta_1)$ and $y = (\rho_2, \theta_2)$ in the cylindrical coordinates of \cref{lem:collar_lemma}, we have by \cref{prop:injradius_estimate} that 
    \begin{equation*}
        \frac{\inj{x,\S}}{\inj{y,\S}} \asymp \frac{\ell(\gamma) \cosh(\rho_1)}{\ell(\gamma) \cosh(\rho_2)} \, .
    \end{equation*}
    Since $\cosh(\rho) \asymp e^{\rho}$ and $\abs{\rho_1 - \rho_2} \leq 1$, the result follows. \end{proof}

Let $M$ be a Riemannian surface and $\gamma :[a,b] \to M$ a unit speed geodesic in $M$. For every $r \in \reals$ and $t \in [a,b]$, let $\varphi_\gamma(r,t)$ be the point obtained by going a signed distance $r$ along the geodesic which is perpendicular to $\gamma$ at the point $\gamma(t)$. If a point $z \in M$ has exactly one pair $(r,t)$ such that $z = \varphi_\gamma(r,t)$, then $(r,t)$ are called the \textbf{Fermi coordinates} of $z$ with respect to $\gamma$.

When $\gamma$ is a geodesic in the hyperbolic plane $\hyperbolic$, the map $\varphi_\gamma: \reals^2 \to \hyperbolic$ is a bijection, and hence it is just a coordinate change; the metric tensor of $\hyperbolic$ in Fermi coordinates is given by 
\begin{equation} \label{eq:fermi_metric_tensor}
    ds^2 = dr^2 + \cosh(r)^2 dt^2 \, ,
\end{equation}
see \cite[Chapter 1]{buser_book}. For a general surface, the map $\varphi_\gamma$ is usually not one-to-one; however, we can still use Fermi coordinates as long as $\varphi_\gamma$ is restricted to domains where it is one-to-one. 

\begin{proposition} \label{prop:fermi_is_possible} Let $M$ be a Riemannian surface whose curvature is bounded from above by a nonnegative constant $K$ and let $\gamma : [0, \ell] \to M$ be a unit speed minimal geodesic. Consider the set  
\begin{equation*}
     D_\gamma = \big \{ (r,t) : t\in [0,\ell] \, , |r| \leq \min \big \{\frac{\inj{\gamma(t)}}{4}, ~\frac{\pi}{2 \sqrt{K}}\big \}\big \} \, .
\end{equation*}
Then the map $\varphi_\gamma:D_\gamma\to M$ is one-to-one.
\end{proposition}
\begin{proof}    
Assume by contradiction that $z = \varphi_\gamma(t,r) = \varphi_\gamma(s,\rho)$ for $t \neq s$. Since $\mdist{z}{\gamma(t)}{\S} \leq \inj{\gamma(t)}/4$ and $\inj{\cdot}$ is $1$-Lipschitz, we deduce that $\inj{z} \geq 3 \inj{\gamma(t)}/4$; similarly, $\inj{z} \geq 3 \inj{\gamma(s)}/4$. Thus, the geodesic triangle $T$ whose points are $z$, $\gamma(t)$ and $\gamma(s)$ is contained in the ball $B_M(z, \inj{z})$. By the triangle inequality, all sides of $T$ have lengths $\leq \pi / \sqrt{K}$. By a standard triangle comparison theorem \cite[][Theorem 2.7.6]{Klingenberg_riemannian_geometry}, every angle of $T$ is smaller than the corresponding angle of a model triangle with the same side-lengths in the sphere of curvature $K$. Since the lengths of the segments connecting $z$ to $\gamma(t)$ and $z$ to $\gamma(s)$ are smaller than $\pi/2\sqrt{K}$, $T$ must have an acute angle on the segment connecting $\gamma(t)$ to $\gamma(s)$; this is a contradiction to the fact that both angles on this segment are right angles.
\end{proof}

\section{Extremal length and the inverse injectivity radius}\label{sec:extremal_length_proof}
The goal of this section is to prove \cref{thm:main_extremal_length}. Let us briefly explain the intuition of the proof. For each $\rho:S \to [0,\infty)$ we need to find a curve from $\partial B_S(x,r_x)$ to $\partial B_S(y,r_y)$ which has small $\rho$-length. An adversarial choice of $\rho$ would put significant weight on the narrow regions of the surface (that is, inside thin collars) since there it can capture all the curves passing through the collar for a relatively small price to pay to the contribution to  $\int_\S \rho^2$. Hence we wish to find curves that try to avoid thin collars. To that aim, the reciprocal injectivity weight (\cref{def:weight}) is useful, since it ``punishes'' curves that go through thin collars. An initial approach to bounding the extremal length is as follows: take a curve $\gamma$ which minimizes the reciprocal injectivity weight among all curves connecting $x$ to $y$ and consider a tubular neighborhood $D$ around $\gamma$ which is as wide as possible so as to not intersect itself; now draw a random curve in $D$ which roughly follows the same trajectory as $\gamma$ and bound its expected $\rho$-length. 

There are several technical obstacles with this approach. First, $\gamma$ need not be differentiable, making analysis in its tubular neighborhood difficult. Second, the tubular neighborhood $D$ can be very far from rectangular; it may go through various narrow and wide regions, fanning in and out many times, making the calculations complicated. To overcome both difficulties, we conformally stretch $S$ so that a collar of width $\ell\ll 1$ turns to a cylinder of width $\Theta(1)$ and length $1/\ell$ while incurring only a constant additive error to the curvature at any point. 
In this new  surface $\mS$ the calculation of extremal length is easier, and since extremal length is a conformal invariant, the bound we obtain holds also for the original manifold $\S$. In \cref{sec:stretch} we show how to perform the stretching and in \cref{sec:proofEL} we bound the extremal length. 

\subsection{A conformal change of metric}\label{sec:stretch} 


Any conformal coordinate change can be expressed by multiplying the metric by a smooth function $f:\S \to (0,\infty)$. Since we would like a collar of width $\ell \ll 1$ to turn into a cylinder of constant width, it makes sense to take $f$ simply as $1/\truncinj{z}$ for all $z \in \S$. However, this function is not differentiable, and moreover, it is hard to control the curvature change that it induces. We will instead take a smooth approximation of it.

Let $\gamma_1, \ldots, \gamma_s$ be the set of all simple closed geodesics of length $\leq 2\sinh^{-1}(1)$ in $\S$. Recall from the collar lemma (\cref{lem:collar_lemma}) that $\S$ can be partitioned into parts, where one part is ``thick'' and has injectivity radius bounded below by $\sinh^{-1}(1)$, while the others (the ``collars'') are isometric to cylinders of length $2W(\gamma_i)$ defined in \cref{lem:collar_lemma}. For every $i = 1 \ldots, s$, define the function $f_i :\reals \to (0,\infty)$ by    
    \begin{equation*}
        f_i(x) = \begin{cases}
            \frac{1}{\ell(\gamma_i) \cosh(x)} & \abs{x} < W(\gamma_i) - 2 \\
            1 & \abs{x} \geq W(\gamma_i) - 2 \,.
        \end{cases}
    \end{equation*}
It is a straightforward calculation that $W(\gamma_i) \in \left(\log(\frac{1}{\ell(\gamma_i)}) + 1, \log(\frac{1}{\ell(\gamma_i)})+2\right)$, and hence the jump discontinuity in $f_i$ is bounded above by a constant. \cref{prop:injradius_estimate} then implies that if $z \in C(\gamma_i)$ has Fermi coordinates $(\rho, \theta)$ relative to the geodesic $\gamma_i$, then $f_i(\rho) \asymp 1/\inj{z}$.
    
Let $\psi(x):\reals \to \reals$ be the standard smooth mollifier
    \begin{equation*}
        \psi(x) = \begin{cases}
            ce^{-\frac{1}{1-x^2}} & \abs{x} \leq 1 \\
            0 & \mathrm{otherwise} \,,
        \end{cases} 
    \end{equation*}
where $c$ is chosen so that $\int_{-1}^{1} \psi(x)dx = 1$. We now define $f: \S \to (0,\infty)$ as follows: 
    \begin{equation}\label{def:f}
        f(z) = \begin{cases}
			1 & \text{if $z \not \in \union_{i}C(\gamma_i)$} \\
            \displaystyle \int_{-W(\gamma_i)}^{W(\gamma_i)} f_i(x)\psi(\rho - x)dx & \text{if $z\in C(\gamma_i)$ and has Fermi coordinate $(\rho, \theta)$} 
            \, .
		 \end{cases}   
   \end{equation}
It is immediate to check that $f$ is smooth. Hence, if $g$ is the hyperbolic metric on $\S$, then the surface $\mS$ defined on the same point-set as $\S$ and equipped with the metric  $\tilde{g} = f(z)g$, is a Riemannian surface conformal to $\S$. We denote the length of curves on $\mS$ by $\tilde{\ell}$, the integral over curves on $\mS$ by $\int \abs{d\tilde{z}}$, the area of sets by $\tilde{\vol}$, and the Fermi coordinate map by $\tilde{\varphi}$. 

\begin{proposition} \label{prop:change_of_metric} The function $f$ and the resulting surface $\mS$ satisfy the following properties.     
    \begin{enumerate}
        \item[$(1)$] \label{enu:f_is_asymp_to_inj} $f(z) \asymp 1/\truncinj{z}$ for all $z \in \S$.
        
        \item[$(2)$] \label{enu:curvature_is_bounded} There exist universal constants $K_1, K_2 \geq 0$ so that  
        $K(z) \in [-K_1, K_2]$ 
        for all $z \in \mS$, where $K : \mS \to \reals$ is the curvature in $\mS$. 
        \item[$(3)$] \label{enu:injectivity_is_constant} $\inj{\mS} \gtrsim 1$.
        
        \item[$(4)$] \label{enu:metric_tensor_is_bounded} There exists a universal constant $r_0>0$ such that the following holds. Let $(r,t)$ be the Fermi coordinates relative to a minimal geodesic $\gamma$ in $\mS$, and assume $\abs{r} \leq r_0$. Then the metric tensor under these coordinates is given by
        \begin{equation*}
            ds^2 = dr^2 + G(r,t)^2 dt^2 \, ,
        \end{equation*}
        where
        \begin{equation}\label{eq:GrtBound}
            \cos(r \sqrt{K_2}) \leq G(r,t) \leq \cosh(r \sqrt{K_1}) \, .
        \end{equation}
    \end{enumerate}
\end{proposition}

\begin{proof}

Firstly, since $\psi$ has finite support and $f_i(x) \asymp f_i(y)$ for all $\abs{x-y} \leq 2$ and all $i$, the smoothing cannot change the value of $f_i$ by more than a constant factor. Hence $f(z) \asymp 1/\truncinj{z}$ for all $z \in \S$ by \cref{prop:injradius_estimate}.
        
Secondly, the function $f$ is constant and equal to $1$ in the thick part of $\S$, as well as in all collars $C(\gamma_i)$ with $W(\gamma_i) \leq 2$, and so the curvature there remains $-1$. It therefore suffices to calculate the new curvature in the collars $C(\gamma_i)$ with $W(\gamma_i)>2$. Recall that the curvature of $\mS$ is given by 
        \begin{equation*}
            K = \frac{1}{f} \left(-1 - \Delta \log(f) \right) \,
        \end{equation*}
(see e.g. \cite[][Theorem 7.30]{lee_introduction_to_riemannian_manifolds}). Let $(\rho, \theta)$ be the cylindrical coordinates of a point $z \in C(\gamma_i)$ as in \cref{lem:collar_lemma}. Observe that $f(z)$ does not depend on $\theta$, so by slight abuse of notation, we may write $f(z) = f(\rho)$ inside the collar. The Laplacian of the $\log$ is then given by 
        \begin{equation*}
            \Delta \log f = \frac{1}{\cosh(\rho)} \frac{\partial}{\partial \rho} \left( \cosh(\rho) \frac{\partial}{\partial \rho} \log(f) \right) \,,
        \end{equation*}
        which yields, after a straightforward calculation,
        \begin{equation*}
            K = -\frac{1}{f} \left(\frac{f^2 + f(\tanh(\rho)f' + f'') - (f')^2 }{f^2} \right) \, .
        \end{equation*}
Since $\psi(x)$ has bounded first and second derivatives, by our choice of $f$ we have that $f(\rho) \asymp f'(\rho) \asymp f''(\rho)$, and so the absolute value of the expression in the parenthesis is bounded above by a constant. Since $f$ is bounded from below by a universal constant, this gives the desired curvature bounds.

Thirdly, by a result of Klingenberg (see e.g. \cite[][Lemma 6.4.7]{petersen_riemannian_geometry}), if a Riemannian surface $M$ has curvature bounded above by a positive constant $K_2$, then  
        \begin{equation*}
            \inj{M} \geq \min \bigg\{\frac{\pi}{\sqrt{K_2}}, \frac{1}{2} \cdot (\text{length of shortest closed geodesic}) \bigg\} \,.
        \end{equation*}
We have just proved above that $\mS$ has curvature bounded above by a universal constant, so the first term in the minimum above is bounded away from $0$. Let $\gamma \subseteq \mS$ be a shortest closed geodesic in $\mS$. Assume first that $\gamma$ is contractible (note that we cannot immediately rule out this case because $\mS$ is no longer of non-positive curvature) and denote its interior by $\Omega$. If $\tilde{\ell}(\gamma) \geq 1$ there is nothing to prove, so we also assume that $\tilde{\ell}(\gamma) \leq 1$. Since $f(z) \gtrsim 1$, it follows that $\ell(\gamma) \lesssim 1$ as well. By a standard isoperimetric inequality \cite[][Theorem V.5.3]{chavel_book2}, we have $\ell(\gamma) \gtrsim \vol(\Omega)^{1/2}$ and since $\ell(\gamma) \lesssim 1$, \cref{lem:weight_is_subexponential} implies that the values of $f(z)$ for all $z \in \gamma \union \Omega$ differ by at most a universal multiplicative constant. Hence $\tilde{\ell}(\gamma) \gtrsim \tilde{\vol}(\Omega)^{1/2}$. On the other hand, since $\gamma$ is a closed geodesic in $\mS$ and $\Omega$ is a topological disc, by the Gauss-Bonnet theorem, $\int_\Omega K(z) d\tilde{z} = 2\pi$ which implies that $\tilde{\vol}(\Omega) \geq 2\pi/K_2$. This shows that $\tilde{\ell}(\gamma) \gtrsim 1$ as wanted.

The case that $\gamma$ is non-contractible is easier. Then $\gamma$'s length can be calculated by integrating over points in $\S$:
        \begin{equation*}
            \tilde{\ell}(\gamma) = \int_\gamma f(z) \abs{dz} \asymp \int_\gamma \frac{1}{\truncinj{z}}\abs{dz} \, .
        \end{equation*}
Since the length of $\gamma$ in $S$ is larger than $\inj{z}$ for all $z\in \gamma$, we obtain $\tilde{\ell}(\gamma) \gtrsim 1$.

Lastly, for the fourth assertion, for every $t$ let $\eta_t$ be the geodesic perpendicular to $\gamma$ at $\gamma(t)$. For a fixed $t_0$, the vector field $\frac{\partial}{\partial t}$ is a Jacobi field along $\eta_{t_0}$ and is perpendicular to it. Hence, the function $G(r) = \abs{\frac{\partial}{\partial t}}_{(r,t_0)}$ satisfies the Jacobi equation $\partial_r^2 G + KG = 0$, with initial conditions $G(0) = 1$ and $\partial_r G(0) = 0$ and so the function $h = \frac{\partial_r G}{G}$ satisfies the Riccati equation $\partial_r h + h^2 +K = 0$. The result then follows from standard comparison estimates (see e.g. \cite[][Proposition 25 and the discussion thereafter]{petersen_riemannian_geometry}); for instance, for the upper bound, the real function $f(r)$ satisfying $f'=-f^2 +K_1$ with initial condition $f(0)=0$ is easily seen to be $\left(\log\left(\cosh \left(\sqrt{K_1}r\right)\right)\right)'$. Since we proved that $K \geq -K_1$, standard comparison arguments yield that $h(r) \leq f(r)$ which, by integration, gives the upper bound in \eqref{eq:GrtBound}. 
\end{proof}

\subsection{Proof of \texorpdfstring{\cref{thm:main_extremal_length}}{}}\label{sec:proofEL} 

Let $f:S\to (0,\infty)$ be the function defined  in \eqref{def:f} and let $\mS$ be the resulting surface defined below it. Since extremal length is conformally invariant, it suffices to upper bound $\el{\Gamma}$ on the new manifold $\mS$ instead of on $\S$. 

Let $\gamma$ be a minimal unit speed geodesic in $\mS$ connecting $x$ to $y$. 
We have that 
\begin{equation} \label{eq:equivLs}
\tilde{\ell}(\gamma) \asymp \wdist{x}{y} \, .
\end{equation}
Indeed, this follows since $f(z) \asymp \truncinj{z}^{-1}$ by \cref{prop:change_of_metric} part $(1)$, and since the left hand side equals $\inf_{\eta \subseteq \S} \int_\eta f(z) \abs{dz}$ and the right hand side equals  $\inf_{\eta \subseteq \S} \int_\eta \frac{1}{\truncinj{z}} \abs{dz}$ where both infima are over all curves $\eta$ connecting $x$ to $y$.  

Our goal is to find, for every $\rho:\mS\to [0,\infty)$, a curve in $\Gamma$ with short $\rho$-length. 
We find such a curve by selecting a random curve contained in a tubular neighborhood $D \subseteq \S$ of $\gamma$. Since it is best to use as ``wide'' a domain as possible, the radius of this tubular neighborhood, while forced to start small, will quickly expand until it reaches the largest possible width, as we now explain. 

Let $\alpha > 0$ be a universal constant small enough so that the following properties hold:
\begin{enumerate}
    \item The Fermi coordinate map $\tilde{\varphi}_{\gamma}$ is one-to-one in the region $[-\alpha, \alpha]\times [0,\tilde{\ell}(\gamma)]$; this is possible by \cref{prop:fermi_is_possible} together with \cref{prop:change_of_metric} parts $(2)$ and $(3)$. 

    \item We have that $\alpha \leq K_2^{-1/2}$ where $K_2\in(0,\infty)$ is the universal constant from \cref{prop:change_of_metric} part $(2)$.
    
    \item The balls $B_{\mS}(x,\alpha \frac{r_x}{\truncinj{x}})$ and $B_{\mS}(y,\alpha \frac{r_y}{\truncinj{y}})$ are  contained in the balls $B_\S(x, r_x)$ and $B_\S(y, r_y)$ respectively; this is possible by \cref{prop:change_of_metric} part $(1)$. 
\end{enumerate}


We now define the tubular neighborhood $D$, which will be given in terms of Fermi coordinates around the geodesic $\gamma$. The construction of $D$ depends on whether $\tilde{\ell}(\gamma)\geq 1$ or not. Let us assume now that $\tilde{\ell}(\gamma)\geq 1$ and deal with the other case later. We define a continuous piecewise differentiable function $r : [0, \tilde{\ell}(\gamma)] \to (0,\infty)$ by 
\begin{equation*}
    r(t) = \begin{cases}            
            \alpha \left(4t + (1-4t) \frac{r_x}{\truncinj{x}} \right) & t \in [0, 1/4] \\
            
            \alpha & t \in [1/4, \tilde{\ell}(\gamma) - 1/4] \\

            \alpha \left( 4(\tilde{\ell}(\gamma)-t) + (1-4(\tilde{\ell}(\gamma)-t))\frac{r_y}{\truncinj{y}} \right) & t \in [\tilde{\ell}(\gamma) - 1/4, \tilde{\ell}(\gamma)] \, .
        \end{cases}
\end{equation*}
Since $r_x \leq \truncinj{x}$ and $r_y \leq \truncinj{y}$ we have that $r(t) \leq \alpha$ and $\abs{r'(t)} \lesssim 1$ for all $t\in[0,\tilde{\ell}(\gamma)]$. Next, for every $q \in [-1,1]$, let $\gamma_{q} : [0, \tilde{\ell}(\gamma)] \to \mS$ be given by $\gamma_{q}(t) = \tilde{\varphi}_{\gamma}(q \cdot r(t),t)$. Lastly, set $D = \{\gamma_{q}(t) \mid q\in [-1,1], t \in [0, \tilde{\ell}(\gamma)] \}$, see \cref{fig:geodesic_fan_in_out_large_ell}.

\begin{figure}[!ht]
        \centering
        \includegraphics[width=0.9 \textwidth]{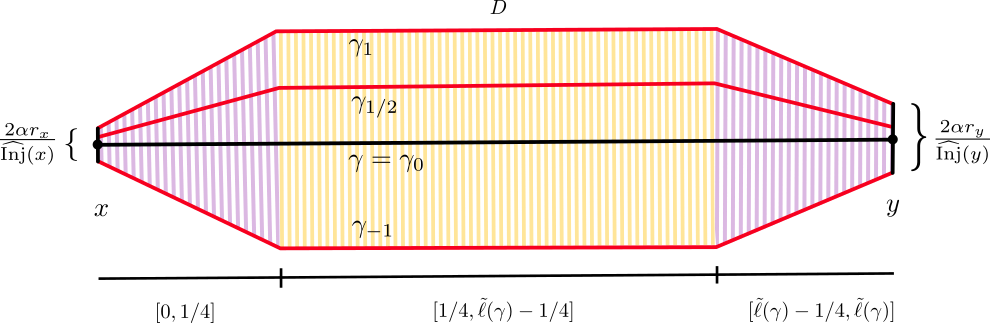}
        \caption{The domain $D$ as a subset of $\tilde{\S}$. Note that the uniform thickness in the middle part is guaranteed by the injectivity radius lower bound of $\tilde{\S}$. The paths $\gamma_q$ start at distances in $[-\alpha r_x/\truncinj{x}, \alpha r_x/\truncinj{x}]$ from $x$, fan out to a spread of $\alpha$, then fan back towards $y$.}
        \label{fig:geodesic_fan_in_out_large_ell}
\end{figure}

            


Denote by $\mu$ the Riemannian area measure on $D$ according to the metric $\tilde{g}$, and by $\nu$ the measure on $D$ of $\gamma_q(t)$, obtained by taking $q$ to be a uniform random variable in $[-1,1]$ and $t$ an independent uniform random variable in $[0,\tilde{\ell}(\gamma)]$. Let $\rho : \mS \to [0,\infty)$ be a Borel measurable function. Integrating $\rho$ according to $\nu$, we have
\begin{equation*}
    \int_{D} \rho ~d\nu = \frac{1}{2 \tilde{\ell}(\gamma)} \int_{-1}^1 \int_0^{\tilde{\ell}(\gamma)}  \rho(\gamma_{q}(t))~dt ~dq \, .
\end{equation*}
Thus, there exists a $q^*\in [-1,1]$ such that
\begin{equation*}
    \int_0^{\tilde{\ell}(\gamma)}  \rho(\gamma_{q^*}(t))~dt \leq \tilde{\ell}(\gamma) \int_D \rho ~ d\nu \, .
\end{equation*}
For each fixed $q\in[-1,1]$ and $t\in[0,\tilde{\ell}(\gamma)]$ we have $|\gamma'_q(t)|^2 = q^2 r'(t)^2 + G(r,t)^2$, where $G(r,t)$ is from \cref{prop:change_of_metric} part $(4)$. Since $\abs{r'(t)} \lesssim 1$ and by \cref{prop:change_of_metric} parts $(2)$ and $(4)$, we get that $|\gamma'_q(t)| \lesssim 1$. Hence,

\begin{equation*}
    \int_{\gamma_{q}} \rho ~\abs{d\tilde{z}} = \int_0^{\tilde{\ell}(\gamma)} \rho(\gamma_{q}(t)) \abs{\gamma_{q}'(t)} ~dt \lesssim \int_0^{\tilde{\ell}(\gamma)} \rho(\gamma_{q}(t)) ~dt \, 
\end{equation*}
and so 
\begin{equation*}
    \int_{\gamma_{q^*}} \rho ~\abs{d\tilde{z}} \lesssim \tilde{\ell}(\gamma) \int_D \rho ~ d\nu \, .
\end{equation*}
We bound $\int_{D} \rho ~d\nu$ by comparing it to $\mu$ using the Cauchy-Schwarz inequality,
\begin{equation*}
    \int_{D} \rho ~d\nu = \int_{D} \rho \frac{d\nu}{d\mu} ~d\mu \leq \sqrt{\int_{D} \rho^2 ~d\mu} ~\sqrt{\int_{D} \left( \frac{d\nu}{d\mu}\right)^2 ~d\mu}  \, .
\end{equation*}
Since every curve $\gamma_{q}$ contains a curve in $\Gamma$, we take $\gamma_{q^*}$ and obtain 
\begin{equation}\label{eq:radon_nikodym_bound_on_extremal_length2}
    \el{\Gamma} = \sup_\rho \inf_{\eta \in \Gamma} \frac{L(\eta, \rho)^2}{\int_{D} \rho^2 ~d\mu} \leq \sup_\rho \inf_{q \in [-1,1]} \frac{\left(\int_{\gamma_{q}} \rho ~\abs{d\tilde{z}} \right)^2}{\int_{D} \rho^2 ~d\mu} \lesssim \tilde{\ell}(\gamma)^2 \int_{D} \left(\frac{d\nu}{d\mu} \right)^2 d\mu \, .
\end{equation}
We now bound the Radon-Nikodym derivative in the right-hand side. For a point $z\in D$, the Radon-Nikodym derivative is given by 
\begin{equation}\label{eq:radon_nikodym_basic}
    \frac{d\nu}{d\mu}(z) = \lim_{\eps \to 0} \frac{\nu\left( B_{\mS}(z,\eps)\right)}{\mu \left( B_{\mS}(z,\eps) \right)} \asymp \lim_{\eps\to 0} \frac{\mathrm{Leb}\left(\{(q,t) \in [-1,1] \times [0,\tilde{\ell}(\gamma)] \mid \gamma_{q}(t) \in B_{\mS}(z, \eps) \} \right)}{ 2 \tilde{\ell}(\gamma) \eps^2}\, ,
\end{equation}
where we use the fact that $\mu(B_{\mS}(z,\eps)) \asymp \eps^2$ since the curvature of $\mS$ is in $[-K_1,K_2]$ (so for small enough $\eps$, $\mu(B_{\mS}(z,\eps))$ is larger than the area of a spherical cap of radius $\eps$ on a sphere of radius $1/\sqrt{K_2}$, but smaller than the volume of a hyperbolic disc of curvature $-K_1$ of radius $\eps$, see e.g. \cite[][Theorems 11.14, 11.19]{lee_introduction_to_riemannian_manifolds}).

\begin{claim} \label{eq:ss_are_close_together}
    Let $z_1, z_2 \in D$ have Fermi coordinates $(r_1, t_1)$ and $(r_2, t_2)$ with respect to $\gamma$ under the metric $\tilde{g}$.
    For any $\eps>0$ small enough, if the distance between $z_1$ and $z_2$ is at most $\eps$, then $\abs{t_1 - t_2} \lesssim \eps$. 
\end{claim}
\begin{proof}

    Let $\eta:[0, \mdist{z_1}{z_2}{\mS}] \to \mS$ be a unit speed minimal geodesic from $z_1$ to $z_2$. Denoting the Fermi coordinates of $\eta(s)$ by $(r(s), t(s))$, we have
    \begin{align*}
        \eps &\geq \mdist{z_1}{z_2}{\mS} = \int_{0}^{\mdist{z_1}{z_2}{\mS}} \sqrt{\left(\frac{dr}{ds}\right)^2 + G(r(s), t(s))^2\left(\frac{dt}{ds}\right)^2} ds \\
        & \geq \int_{0}^{\mdist{z_1}{z_2}{\mS}} \cos(r(s)\sqrt{K_2}) \frac{dt}{ds}ds  \gtrsim \abs{t_2 - t_1} \, , 
    \end{align*}
where the second inequality is by  \cref{prop:change_of_metric} part $(4)$ and the last inequality is since we chose $\alpha \leq K^{-1/2}$ . 
\end{proof}

Fix $z\in D$ and let $t$ be such that there exists a $q \in [-1,1]$ with $\gamma_{q}(t) \in B_{\mS}(z,\eps)$. For every such $t$, the curve $\eta_t:[-1,1] \to \mS$ given by $\eta_t(q) = \tilde{\varphi}_\gamma(q\cdot r(t), t)$ is a minimal geodesic with speed $r(t)$. Since it is minimal, the length of $\eta_t \intersect B_{\mS}(z, \eps)$ is bounded by $2\eps$, and so the Lebesgue measure of all $q$ such that $\eta_t(q) \in B_{\mS}(z,\eps)$ is bounded by $2\eps / r(t)$. Denoting the Fermi coordinates of $z$ by $(r_z, t_z)$ in $\mS$ relative to $\gamma$, by \cref{eq:ss_are_close_together} and the fact that $r(t)$ is continuous and with bounded derivative, for small enough $\varepsilon$,
\begin{equation*}
    \mathrm{Leb}\left(\{(q,t) \in [-1,1] \times [0,\tilde{\ell}(\gamma)] \mid \gamma_{q}(t) \in B_{\mS}(z, \eps) \} \right) \lesssim \int_{t_z - C\eps}^{t_z + C\eps} \frac{\eps}{r(t)} ~dt \asymp \frac{\eps^2}{r(t_z)} \, .
\end{equation*}
We deduce by \eqref{eq:radon_nikodym_basic} that $\frac{d\nu}{d\mu}(z) \lesssim 1/(r(t_z) \tilde{\ell}(\gamma))$.  Combining this with \eqref{eq:radon_nikodym_bound_on_extremal_length2} and \cref{prop:change_of_metric} part $(4)$ we get
\begin{align}
    \el{\Gamma} \lesssim \int_D \frac{d\mu(z)}{r(t_z)^2}  \leq \int_{0}^{\tilde{\ell}(\gamma)} \int_{-r(t)}^{r(t)} \frac{\cosh(r(t)\sqrt{K_1})}{r(t)^2} ds dt 
    \lesssim \int_{0}^{\tilde{\ell}(\gamma)}\frac{1}{r(t)} ~dt \label{eq:el_as_integral} \,.
\end{align}
We partition the integral over $t$ in \eqref{eq:el_as_integral} into three parts as in the definition of $r(t)$. Firstly, when $t\in [1/4, \tilde{\ell}(\gamma)-1/4]$ we have that $r(t)=\alpha$, and so
\begin{equation}\label{eq:weighted_distance_term_in_el}
    \int_{1/4}^{\tilde{\ell}(\gamma)-1/4} \frac{1}{r(t)} ~dt \lesssim \tilde{\ell}(\gamma)  \asymp  \wdist{x}{y}  \, ,
\end{equation}  
by \eqref{eq:equivLs}. Secondly, when $t \in [0,1/4]$, we get    \begin{equation}\label{eq:logx_term_in_el}
        \int_{0}^{1/4}\frac{1}{r(t)}~dt = \int_{0}^{1/4} \frac{1}{\alpha \left( 4t + (1-4t) \frac{r_x}{\truncinj{x}}\right)} ~dt =  \frac{\log \left(\frac{\truncinj{x}}{r_x} \right)}{4 \alpha \left(1-\frac{r_x}{\truncinj{x}}\right)} \lesssim \log \left(\frac{\truncinj{x}}{r_x} \right)\, ,
    \end{equation}
since $r_x \leq \truncinj{x}/2$. A similar calculation shows that when $t \in [\tilde{\ell}(\gamma)-1/4, \tilde{\ell}(\gamma)]$, we have  
\begin{equation}\label{eq:logy_term_in_el}
    \int_{\tilde{\ell}(\gamma)-1/4}^{\tilde{\ell}(\gamma)} \frac{1}{r(t)}~dt \lesssim \log \left(\frac{\truncinj{y}}{r_y} \right)\, .
\end{equation}
Putting \eqref{eq:weighted_distance_term_in_el}, \eqref{eq:logx_term_in_el} and \eqref{eq:logy_term_in_el} into \eqref{eq:el_as_integral} proves the desired inequality. 

Lastly, we treat the  $\tilde{\ell}(\gamma) < 1$. Set
\begin{equation*}
    r(t) = \begin{cases}            
            \alpha \left(2t + (\tilde{\ell}(\gamma)-2t) \frac{r_x}{\truncinj{x}} \right) & t \in [0, \tilde{\ell}(\gamma)/2] \\

            \alpha \left( 2(\tilde{\ell}(\gamma)-t) + (\tilde{\ell}(\gamma)-2(\tilde{\ell}(\gamma)-t))\frac{r_y}{\truncinj{y}} \right) & t \in [\tilde{\ell}(\gamma)/2, \tilde{\ell}(\gamma)] \, ,
        \end{cases}
\end{equation*}
see \cref{fig:geodesic_fan_in_out_small_ell}. The reasoning that led to \eqref{eq:el_as_integral} holds in this case as well, and we need only calculate the integral on its right-hand side. When $t \in [0, \tilde{\ell}(\gamma)/2]$, we get 
\begin{align*}
    \int_0^{\tilde{\ell}(\gamma)/2} \frac{1}{r(t)}dt = \int_0^{\tilde{\ell}(\gamma)/2}\frac{1}{\alpha\left(2t+\left(\tilde{\ell}\left(\gamma\right)-2t\right)\frac{r_{x}}{\truncinj{x}}\right)} 
    \lesssim \log \left( \frac{\truncinj{x}}{r_x}\right) \, .
\end{align*}
A similar calculation shows that $$ \int_{\tilde{\ell}(\gamma)/2}^{\tilde{\ell}(\gamma)}\frac{1}{r(t)} \lesssim \log \left( \frac{\truncinj{y}}{r_y}\right) \, ,$$
again leading to the desired inequality.
\begin{figure}[!ht]
        \centering
        \includegraphics[width=0.5 \textwidth]{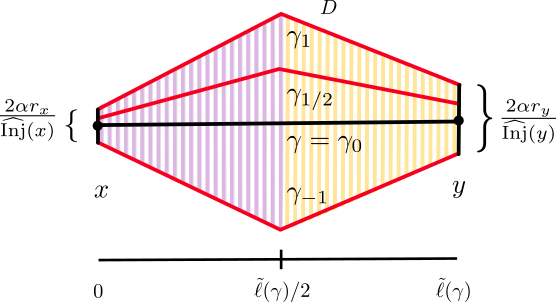}
        \caption{The domain $D$ as a subset of $\tilde{\S}$. In this case there is no  middle part corresponding to the main term $\wdist{x}{y}$ in \eqref{eq:main_extremal_length}.        
        }
        \label{fig:geodesic_fan_in_out_small_ell}
\end{figure} \qed

\section{Proof of the eigenvalues lower bound}\label{sec:lower_bound_proof}

The goal of this section is to prove \cref{thm:main_eigenvalue_lower_bound}. Let $S$ be a compact hyperbolic surface of genus $g$ and let $\{\phi_0, \phi_1, \ldots\}$ be an orthonormal basis of smooth Laplacian eigenvectors, with $\phi_k$ corresponding to $\lambda_k$. See e.g.~\cite[][Theorem 7.2.6]{buser_book} for the existence of such a basis. Fix $\lambda>0$ and define the function $e_\lambda : \S^2 \to \reals$ by
\begin{equation*}
    e_\lambda(x,z) = \sum_{k \geq 1 : \lambda_k \leq \lambda} \phi_k(x)  \phi_k(z) \, ;
\end{equation*}
the function $e_\lambda$ is called the \emph{spectral kernel}. For $x\in \S$, we define 
\begin{equation*}
     \mu_x(\lambda) = e_\lambda(x,x) = \sum_{k \geq 1 : \lambda_k \leq \lambda} \phi_k(x)^2 \, .
\end{equation*} 
For $x$ with $\mu_x(\lambda) > 0$, we define the function $f_x :\S \to \reals$ by
\begin{equation*}
    f_x(z) = \frac{e_\lambda(x,z)}{\sqrt{\mu_x(\lambda)}} \, ,
\end{equation*}
and note that we suppressed $\lambda$ from the notation $f_x$. Observe that $f_x(x) = \sqrt{\mu_x(\lambda)}$ and that
\begin{equation}\label{eq:norm_of_f}
    \norm{f_x}_2^2 = \frac{\int_\S e_\lambda(x,z)^2~dz}{\mu_x(\lambda)} =  \frac{\int_\S \left(\sum_{k \geq 1 : \lambda_k \leq \lambda} \phi_k(x) \cdot \phi_k(z) \right)^2 ~dz}{\mu_x(\lambda)} = \frac{\mu_x(\lambda)}{\mu_x(\lambda)} = 1 \, .
\end{equation}

We will require an upper bound on the gradient of $f$ similar to Bernstein inequalities for Laplacian eigenfunctions (see e.g.~\cite{DM20}). We will use the following bound on the gradient of $f_x$; it follows from known methods (such as \cite{ortega_pridhnani_gradient_bound}) but we were unable to find a reference that implies it, therefore we provide its short proof in \cref{sec:gradientUpper}.

\begin{lemma}\label{lem:gradient_bound} For all $z \in \S$ and $\lambda>0$ we have
    \begin{equation*}
      \abs{\grad f_x(z)} \lesssim  e^{2\sqrt{\lambda}} \truncinj{z}^{-1} \, .  
    \end{equation*}        
\end{lemma}

Recall \cref{def:weight} for the weighted distance $\wdist{\cdot}{\cdot}$. For any $x\in \S$ and $R>0$ let 
\begin{equation*}
    \wball{x}{R} = \{ y \in \S \mid \wdist{x}{y} \leq R\} \, .
\end{equation*}

\begin{lemma}\label{lem:x_with_large_weighted_balls_is_good} There exists a universal constant $C\in(0,\infty)$ such that for any  $R \geq 1$, any $x \in \S$ and any $\lambda\in(0,1]$ the following holds: if either $\vol(\wball{x}{R}) \geq \frac{16}{\mu_x(\lambda)}$ or $\wball{x}{R} = \S$, then  
\begin{equation*}
    \mu_x(\lambda) \leq C \lambda \Big (R + \log\big (\frac{1}{\mu_x(\lambda)\wedge \frac{1}{2}}\big )\Big )  \, .
\end{equation*}
\end{lemma}

\begin{proof} We first find a point $y\in \wball{x}{R}$ such that $f_x(y) \leq \frac{1}{4}f_x(x)$. Indeed, if $\wball{x}{R} = \S$, then it is immediate that there exists $y\in \wball{x}{R}$ with $f_x(y)=0$ since $\int_y f_x(y)d\S=0$. Otherwise, we have that
\begin{equation}
    \vol \big (\{z \in \S \mid f_x(z) > \frac{1}{4}f_x(x) \} \big) < \frac{||f_x||^2}{f_x(x)^2/16} = \frac{16}{\mu_x(\lambda)} \, ,
\end{equation} 
hence in the case $\vol(\wball{x}{R}) \geq \frac{16}{\mu_x(\lambda)}$ we again obtain such a point $y$. 

Since $\lambda \leq 1$, by \cref{lem:gradient_bound} and \cref{lem:weight_is_subexponential}  there exists a universal constant $C'\in (0,\infty)$ such that $\abs{\grad f_x(z)}\leq C'/\truncinj{x}$  for all $z\in B_\S(x,\truncinj{x}/2)$ and $\abs{\grad f_x(z)}\leq C'/\truncinj{y}$ for all $z\in B_\S(x,\truncinj{y}/2)$. We choose $c=\frac{1}{4C'} \wedge \frac{1}{2}$ and set 
\begin{equation}
   r_x = c (f_x(x) \wedge \frac{1}{2}) \truncinj{x}  \, ,
\end{equation}
to obtain that 
\begin{equation*}
   \inf_{z\in B_\S(x,r_x)}  f_x(z)  \geq {3 f_x(x) / 4} \, .
\end{equation*}
Similarly, setting 
\begin{equation*}
 r_y = c (f_x(x)\wedge \frac{1}{2}) \truncinj{y} \, ,
\end{equation*}
yields that
\begin{equation*}
   \sup_{z \in B_\S(y,r_y)}  f_x(z)  \leq f_x(x) /2 \, . 
\end{equation*}
In particular, the balls $B_\S(x,r_x)$ and $B_\S(y,r_y)$ are disjoint. Let $\Gamma$ denote the set of all rectifiable curves between the boundaries of the balls which do not enter the balls. Since the two balls are disjoint and $r_x \leq \truncinj{x}/2$ and $r_u \leq \truncinj{y}/2$ we obtain by  \cref{thm:main_extremal_length} that
\begin{equation*}
   \el{\Gamma} \lesssim R + \log \left( \frac{1}{c(f_x(x)\wedge \frac{1}{2})} \right) \lesssim R + \log \left( \frac{1}{f_x(x)\wedge \frac{1}{2}} \right) \, ,
\end{equation*}
where in the last inequality we absorbed $c$ into the implicit constant, which we may do since $R\geq 1$ and $\log ( 1/(f_x(x)\wedge 1/2))\geq \log 2$. We now let $A$ and $B$ be the level sets $\{ z : f_x(z) \leq f_x(x)/2 \}$ and $\{z : f_x(z) \geq 3f_x(x)/4\}$, respectively. Since $A$ and $B$ contain $B_\S(x,r_x)$ and $B_\S(y,r_y)$, respectively, the extremal length of the set of curves between $A$ and $B$ is not larger than $\el{\Gamma}$.

It is well known that given two disjoint open sets $A$ and $B$ of $\S$, the extremal length of the family of curves connecting $\partial A$ and $\partial B$ in $\S \setminus (A\cup B)$ is the inverse of the Dirichlet energy $\int_S \abs{\grad f}^2$ of the unique harmonic function  $f$ in $\S \setminus (A\cup B)$ taking $1$ on $\partial A$ and $0$ on $\partial B$, see \cite[Lemma III1.1]{MardenRodin66}. Since harmonic functions minimize the Dirichlet energy among all functions with the same boundary values, the function $f_x$ has larger Dirichlet energy, so after rescaling the boundary values we obtain that
\begin{equation}\label{eq:EnergyLowerBound} \int_\S \abs{\grad f_x(z)}^2 ~dz \geq \frac{f_x(x)^2/16}{\el{\Gamma}} \gtrsim \frac{f_x(x)^2}{R + \log( \frac{1}{f_x(x)\wedge \frac{1}{2}})} \, .
\end{equation}
Since $f_x$ has unit norm and is a linear combination of eigenfunctions with corresponding eigenvalues at most $\lambda$ we may bound
\begin{equation*}
   \lambda \geq \left<\Delta f_x, f_x \right>  = \left<\grad f_x, \grad f_x \right> = \int_\S \abs{\grad f_x(z)}^2 ~dz \, . 
\end{equation*}
We put this and $f_x(x)^2=\mu_x(\lambda)$ in \eqref{eq:EnergyLowerBound} and obtain that 
\begin{equation*}
   \mu_x(\lambda) \lesssim \lambda \Big (R + \log\big (\frac{1}{\mu_x(\lambda)\wedge \frac{1}{2}}\big )\Big ) \, , 
\end{equation*}
as required.
\end{proof}

\begin{lemma}\label{lem:integral_over_weighted_balls}
    Let $x \in \S$ and $R \geq 1/2$ be such that $\wball{x}{R} \neq \S$. Then
    \begin{equation*}
        \int_{\wball{x}{R}}\frac{1}{\truncinj{z}^2} ~dz \gtrsim R \, .
    \end{equation*}
\end{lemma}
\begin{proof}
Since $\wball{x}{R} \neq \S$ the boundary $\partial \wball{x}{R}$ is non-empty. Let $y$ be a point in $\partial \wball{x}{R}$ which minimizes the hyperbolic distance $\mdist{x}{y}{\S}$. We denote by $\gamma$ a minimal unit speed geodesic in $\S$ between $x$ and $y$, which is contained in $\wball{x}{R}$ by the choice of $y$. 

We proceed by analyzing two cases. First, if $\ell=\mdist{x}{y}{\S} \leq 1$, then all points $z\in B_\S(x,\ell)$ have $\truncinj{z} \asymp \truncinj{x}$ uniformly by \cref{lem:weight_is_subexponential}. Since $y \in \partial \wball{x}{R}$ we must have that $\w{\gamma} = R$, hence $\truncinj{x}^{-1} \gtrsim R/\ell$. We get
\begin{equation*}
    \int_{\wball{x}{R}}\frac{1}{\truncinj{z}^2} ~dz \geq \int_{B_\S(x,\ell/2)} \frac{1}{\truncinj{z}^2} ~dz \gtrsim \frac{\ell^2}{\truncinj{x}^2} \gtrsim R^2 \geq R/2 \, ,
\end{equation*}
where the first and last inequalities are since $R \geq 1/2$.

Next we handle the case $\ell>1$. Let $c\in(0,1/4)$ be a small universal constant that will be chosen later. Parameterizing $\gamma$ by length, we first note that for any $t \in [0,\ell-c]$ we must have  
\begin{equation}\label{eq:BallContainment}
    B_\S(\gamma(t), c\truncinj{\gamma(t)}) \subseteq \wball{x}{R} \, ,   
\end{equation}
since otherwise the ball $B_\S(\gamma(t), c\truncinj{\gamma(t)})$ contains some other point of $\partial \wball{x}{R}$ which we can connect to $\gamma(t)$ in a path of length at most $c$, contradicting the minimality of $\ell$.

Consider the set $D = \{ (t,r) : t \in [0,\ell-c], |r|\leq c\truncinj{\gamma(t)}\}$ which is a subset of $D_\gamma$ defined in \cref{prop:fermi_is_possible}. Since $\gamma$ is a unit speed minimal geodesic connecting two points, \cref{prop:fermi_is_possible} implies that the Fermi coordinate map $\varphi_\gamma$ is one-to-one. By \eqref{eq:BallContainment} we have $\varphi_\gamma(D) \subseteq \wball{x}{R}$ so
\begin{eqnarray} \label{eq:finalLowerBound}
    \int_{\wball{x}{R}}\frac{1}{\truncinj{z}^2} ~dz &\geq& \int_{\varphi_\gamma(D)} \frac{1}{\truncinj{z}^2} ~dz  = \int_{0}^{\ell-c} \int_{-c\truncinj{\gamma(t)}}^{c\truncinj{\gamma(t)}} \frac{1}{\truncinj{\varphi_\gamma(t,r)}^2} \cosh(r) dr dt \nonumber \\
    &\gtrsim& \int_{0}^{\ell-c} \frac{c}{\truncinj{\gamma(t)}}dt = c \cdot \left(\w{\gamma} - \w{\gamma[\ell-c,\ell]} \right) \,.
\end{eqnarray}
where we used \eqref{eq:fermi_metric_tensor} and the fact that $\truncinj{\varphi_\gamma(t,r)}\asymp \truncinj{\gamma(t)}$.

Let $z=\gamma(\ell-c)$, so that $z\in \wball{x}{R}$ and $d_S(x,z) \geq 1/2$. It follows that if $c$ is small enough, then $\inj{z} \geq {\sqrt{c}/R}$. Indeed, otherwise, 
by \cref{lem:weight_is_subexponential}, any curve $\eta$ connecting $x$ to $z$ must have $\w{\eta} \gtrsim \truncinj{z}^{-1} \geq R/\sqrt{c}$, just by considering its last tip of length $1/2$, which is a contradiction if $c > 0$ is small enough since $\wdist{x}{z} \leq R$. By \cref{lem:weight_is_subexponential} we have $\inj{\gamma(t)} \gtrsim \sqrt{c}/R$ for all $t\in [\ell-c,\ell]$ hence $\w{\gamma[\ell-c,\ell]}  \lesssim \sqrt{c} R$. This with \eqref{eq:finalLowerBound} and $\w{\gamma}= R$ concludes the proof.

\end{proof}

\begin{lemma}\label{lem:smallball_has_small_volume}
    For every $\eps \in (0,1)$ and $R \geq 1$, let 
    \begin{equation*}
        \smallball{R}{\eps} = \{x\in \S \mid \vol \left( \wball{x}{R} \right) < \eps R \tand \wball{x}{R} \neq \S \} \, .
    \end{equation*}
    Then
    \begin{equation*}
        \vol \left(\smallball{R}{\eps}\right) \lesssim \eps \aw(\S) \vol(\S) \, .
    \end{equation*}
\end{lemma}
\begin{proof}
Let $x_1, \ldots, x_m$ be a maximal set of points in $\smallball{R}{\eps}$ such that the interiors of $\wball{x_i}{R/2}$ are pairwise disjoint. By maximality, for every $x \in \smallball{R}{\eps}$ the set $\wball{x}{R/2}$ intersects some $\wball{x_i}{R/2}$, so the triangle inequality implies that $\wdist{x}{x_i} \leq R$ and we deduce that the balls $\{\wball{x_i}{R}\}_{i=1}^m$ cover $\smallball{R}{\eps}$. As the volume of each ball is at most $\eps R$, we have that $\vol(\smallball{R}{\eps}) \leq \eps m R$. On the other hand, 
\begin{equation*}
    \vol(\S) \cdot \aw(\S) = \int_{\S} \frac{1}{\truncinj{z}^2} ~dz \geq \sum_{i=1}^m \int_{\wball{x_i}{R/2}}\frac{1}{\truncinj{z}^2} ~dz \quad \gtrsim \quad m R,
\end{equation*}    
by \cref{lem:integral_over_weighted_balls}, giving the desired result.
\end{proof}

\subsection{Proof of \texorpdfstring{\cref{thm:main_eigenvalue_lower_bound}}{}}

Let $C\in(0,\infty)$ be the universal constant from \cref{lem:x_with_large_weighted_balls_is_good}, set $c=\frac{1}{2C}$ and let $\lambda \in (0,1/(4C)^2)$ be fixed. For each $j \geq 1$ we set the numbers
\begin{equation*}
   \mu_j = 2^j \sqrt{\lambda \aw(\S)} \, , \qquad R_j = c 2^j 
\sqrt{\frac{\aw(\S)}{\lambda}} = c \mu_j / \lambda \, .
\end{equation*}
This choice of parameters guarantees that 
\begin{equation}\label{eq:ineq}
 C \lambda R_j + C\lambda \log\big (\frac{1}{\mu_j \wedge \frac{1}{2}}\big ) < \mu_j \qquad \forall \, \lambda \leq \frac{1}{(4C)^4} \, .
\end{equation}
To see this, note first that $C\lambda R_j=\mu_j/2$ and that 
$$ C\lambda \log\big (\frac{1}{\mu_j \wedge \frac{1}{2}}\big ) < \frac{2C \lambda }{(\mu_j \wedge \frac{1}{2})^{1/2}} \leq 2C \lambda^{3/4} \leq \sqrt{\lambda}/2 \leq \mu_j/2 \, ,$$
where the first inequality holds since $\log(a) < 2\sqrt{a}$ when $a\geq 2$, the second inequality holds since $\mu_j \geq \sqrt{\lambda}$ (recall that $\aw(\S)\geq 1$) and the third holds since $\lambda \leq (4C)^{-4}$. 

We now define for each $j\geq 1$ the sets
$$ \good{j} = \{\ x \in \S \mid \mu_j \leq \mu_x(\lambda) \leq \mu_{j+1} \}  \, ,$$
and proceed to upper bound their volumes. If $x\in G_j$, then $\mu_x(\lambda)\geq \mu_j$ and so \eqref{eq:ineq} implies that the  conclusion of \cref{lem:x_with_large_weighted_balls_is_good} cannot hold. Therefore, $x\in G_j$ implies that both $\vol(\wball{x}{R_j}) \leq 32/\mu_j$ and that $\wball{x}{R_j} \neq \S$. Thus,  \cref{lem:smallball_has_small_volume} with $\eps =\frac{32}{\mu_j R_j} = \frac{32}{c4^j \aw(S)}$ and $R=R_j$ gives
\begin{equation}\label{eq:vol}
    \vol(G_j) \lesssim 4^{-j} \vol(\S) \, .
\end{equation} 

Let $N(\lambda)$ be the number of Laplacian eigenvalues in $(0,\lambda]$. Observe that
\begin{equation*}\label{eq:Nlambda}
N(\lambda)= \int_\S \mu_x(\lambda) ~dx  \, .
\end{equation*}
We upper bound this integral by partitioning $S$ to $\{x : \mu_x(\lambda)<\mu_1\}$ and the union of the $G_j$'s, so that 
\begin{equation}
    N(\lambda)  \leq \mu_1 \vol(\S) + \sum_{j \geq 1} \mu_{j+1} \vol(G_j) \lesssim \sum_{j \geq 0} 2^{-j} \sqrt{\lambda \aw(\S)} \vol(\S) \, ,
\end{equation}
where the last inequality is due to \eqref{eq:vol} and our choice of $\mu_j$. 

We conclude that there exists universal constants $B,b\in(0,\infty)$ such that 
\begin{equation*}
 N(\lambda)< B\sqrt{\lambda \aw(\S)}\vol(\S) \textrm{ for any } \lambda \in (0,b] \, . 
\end{equation*}
If $\lambda >0$ is such that $B\sqrt{\lambda \aw(\S)}\vol(\S)=k$ for some positive integer $k$, then $\lambda\leq b$ if and only if $k\leq B\sqrt{b \aw(\S)}\vol(\S)$. For such $\lambda$ we have $N(\lambda)<k$, hence $\lambda_k \geq \lambda$. We deduce that
$$ \lambda_k \geq \frac{k^2}{B^2 \aw(\S) \vol(\S)^2} \geq \frac{k^2}{(4\pi B)^2 \aw(\S) g^2} \, ,$$
for any positive integer $k\leq B\sqrt{b \aw(\S)}\vol(\S)$; the last inequality holds since  $\vol(S)=2\pi(2g-2)$. Since $I(S)\geq 1$ and $\vol(\S)\geq g$ we obtain the desired lower bound on $\lambda_k$ for all $k \leq  B\sqrt{b}g$; for all other integers $k \in (B\sqrt{b}g, 2g-3)$ we simply bound $\lambda_k \geq \lambda_{k_0}$ with $k_0=\lfloor B\sqrt{b}g \rfloor$. Since $k_0/k\gtrsim 1$ this concludes the proof.
 \qed

\subsection{Proof of the gradient upper bound} \label{sec:gradientUpper}
\begin{proof}[Proof of \cref{lem:gradient_bound}]
We follow the harmonic extension method of \cite[][Section 3]{ortega_pridhnani_gradient_bound}.
Let $\phi_0, \phi_1, \ldots, \phi_k$ be the orthonormal eigenvectors of the Laplacian with eigenvalue smaller than or equal to $\lambda$. Then $f_x$ can be written as the linear combination  $f_x(z) = \sum_{i=1}^k \beta_i \phi_i(z)$, with $\beta_i=\phi_i(x)/\sqrt{\mu_x(\lambda)}$. Let $h(z,t)$ be the harmonic extension of $f_x$ to the manifold $M = \S\times \reals$, i.e. 
\begin{equation*}
    h(z,t) = \sum_{i=1}^k \beta_i \phi_i(z)e^{\sqrt{\lambda_i} t} \, .
\end{equation*}
Observe that $\grad h(z,0)$ agrees with $\grad f_x$ on the coordinates corresponding to $\S$, and has an additional coordinate equal to $\sum_{i=1}^k \beta_i \phi_i(z) \sqrt{\lambda_i}$.  Thus, $|\grad f_x(z)|^2 \leq |\grad h(z,0)|^2$. Since $h$  is harmonic, by a classical theorem of Schoen and Yau (\cite[][Corollary 3.2, p21]{Schoen_Yau_Book} with $K=1$ and $a=1$, see also  (3.3) in \cite{ortega_pridhnani_gradient_bound}) we have
\begin{equation*}
    |\grad h(z,0)|^2 \leq C \sup \{h(y)^2 \mid y \in {B_M\left((z,0), 1 \right)}\} \, .
\end{equation*}
Suppose that the supremum is attained at the point $(z^*,t^*) \in B_M\left((z,0),1 \right)$. Since $h$ is harmonic, $h^2$ is subharmonic, and for all $r < \inj{(z^*,t^*), M}$ we have by \cite[][Theorem 6.2, p77]{Schoen_Yau_Book} that
\begin{equation}\label{eq:subharmonic_mean_value}
    h(z^*,t^*)^2 \leq \frac{1}{\vol\left(B_M\left((z^*, t^*),r \right) \right)} \int_{B_M\left((z^*, t^*),r \right) } h(y)^2 ~dM \, .
\end{equation}
 By the triangle inequality, the ball $B_M\left((z^*,t^*),r \right)$ contains the set $B_\S(z^*, r/2) \times [t^*-r/2, t^*+r/2]$, thus, its volume is at least $4\pi \sinh(r/2)^2 \cdot r/2 \gtrsim r^3$. Together with \eqref{eq:subharmonic_mean_value}, this gives
\begin{equation} \label{eq:gradient_by_subharmonicity}
    |\grad f_x(z)|^2 \lesssim \frac{1}{r^3}\int_{B_M\left((z^*, t^*),r \right) } h(y)^2 dy \, .
\end{equation}
Now take $r = \min\{1, \inj{(z^*,t^*), M}\} = \truncinj{z^*}$. By \cref{lem:weight_is_subexponential}, since $\mdist{z^*}{z}{\S} \leq 1$, we have $\truncinj{z}^{-1} \asymp \truncinj{z^*}^{-1}$, and so $1/r \asymp \truncinj{z}^{-1}$. To bound the integral, observe that $B_M\left((z^*, t^*),r \right) \subseteq \S \times [t^*-r,t^*+r]$.  We then have
\begin{align*}
\int_{B_M\left((z^*, t^*),r \right) } h(y)^2 ~dM &\leq \int_{t^*-r}^{t^*+r} \int_\S h(z,t)^2 ~dz dt \\
    &= \int_{t^*-r}^{t^*+r} \int_\S \left(\sum_{i=1}^k \beta_i \phi_i(z) e^{\sqrt{\lambda_i}t} \right)^2 ~dz dt \\
    &= \int_{t^*-r}^{t^*+r} \sum_{i=1}^k \beta_i^2 e^{2\sqrt{\lambda_i}t} dt \\
    &\leq 2r e^{4\sqrt{\lambda}} \sum_{i=1}^k \beta_i^2 = 2r e^{4\sqrt{\lambda}} \norm{f}_2^2 = 2re^{4\sqrt{\lambda}} \, ,
\end{align*}
where the last equality is due to \eqref{eq:norm_of_f}. Using \eqref{eq:gradient_by_subharmonicity} we deduce that
\begin{equation*}
    \abs{\grad f_x(z)} \lesssim e^{2\sqrt{\lambda}} r^{-1} \lesssim e^{2\sqrt{\lambda}} \truncinj{z}^{-1} \, 
\end{equation*}
as needed.
\end{proof}

\section{Heat kernel bound} \label{sec:heat_kernel_proof}
The goal of this section is to prove \cref{thm:main_heat_kernel}. 
Let $\gamma_1, \ldots, \gamma_s$ be the set of all simple closed geodesics of length $\leq 2\sinh^{-1}(1)$ and let $C(\gamma_i)$ be their collars, as described in \cref{lem:collar_lemma}. It will be convenient to chop off the parts of the collar at distance at most $1$ to its boundary, that is, we define $\truncated{C}(\gamma_i) := \{x \in C(\gamma_i) \mid \mathrm{dist}(x, \partial C(\gamma_i)) \geq 1\}$. 

Let $\thin{\S} := \union_{i} \{\truncated{C}(\gamma_i)\}$ and $\thick{\S} := \S \backslash \thin{\S}$. By \cref{lem:collar_lemma} we have that $\inj{x} \gtrsim 1$ for every $x\in \thick{\S}$ and $\inj{x} \lesssim 1$ for every $x\in \thin{\S}$.

\begin{proposition}\label{prop:counting_argument_by_injradius} 
For any $x\in \thick{\S}$ and any $t \in [1/2,1]$ we have 
\begin{equation*}
    \mheat{t}{x}{x}{\S} \lesssim \mheat{t}{0}{0}{\hyperbolic} + 1 \, .
\end{equation*}
\end{proposition}

\begin{proof}
Recall that the heat kernel $\mheat{t}{x}{y}{\S}$ of any hyperbolic surface $\S$ is given by $\mheat{t}{x}{y}{\S} = \sum_{T \in \Gamma} \mheat{t}{\Tilde{x}}{T\Tilde{y}}{\hyperbolic}$, where $\Gamma$ is a group of isometries acting on $\hyperbolic$ so that $\S = \Gamma \backslash \hyperbolic$, and $\Tilde{x}$ and $\Tilde{y}$ are arbitrary inverse images of $x$ and $y$ of the projection to $\S$. To estimate $\mheat{t}{x}{x}{\S}$, we partition the group translations by distance of $T \tilde x$ from $\tilde x$.  Denoting
\begin{equation*}
    \Gamma(m) := \{T \in \Gamma \mid m < d_\hyperbolic(\tilde x, T \tilde x) \leq m+1 \} \, ,
\end{equation*} 
we have
\begin{align}
   \mheat{t}{x}{x}{\S}&= \mheat{t}{0}{0}{\hyperbolic} + \sum_{m=0}^{\infty} \sum_{T \in \Gamma(m)} \mheat{t}{\tilde x}{T\tilde x}{\hyperbolic} \, .
   \label{eq:partition_of_the_count}     
\end{align}
We bound $|\Gamma(m)|$ by an explicit version of \cite[][Lemma 7.5.3]{buser_book}. Let $D$ be the disc around $\tilde x$ of radius $\inj{x}$. Every disc $T(D)$ with $T\in \Gamma(m)$ is contained in the disc of radius $m+1+\inj{x}$ around $\tilde x$ in $\hyperbolic$ which has volume $2\pi \left( \cosh(m+1+\inj{x})-1 \right) \leq 2\pi e^{m+1+\inj{x}}$. On the other hand, the volume of $T(D)$ is $2\pi \left( \cosh(\inj{x})-1 \right)$ and $\{T(D)\}_{T\in \Gamma(m)}$ are non-intersecting. Thus,
\begin{equation}\label{eq:counting_group_actions}
    \#\Gamma(m) \leq \frac{e^{\inj{x}+1}}{\cosh(\inj{x})-1} e^m \lesssim \left(1+\frac{1}{\inj{x}^2} \right) e^m \lesssim e^m \, .
\end{equation}
where in the second inequality we relied on the asymptotics of the function $\cosh(r)$ for $r \to 0$ and $r \to \infty$, and in the third inequality we relied on our assumption that $x\in \thick{\S}$. Next, a well-known bound on $p_t^{\hyperbolic}$ (see \cite[][Lemma 7.4.26]{buser_book}) states that  
\begin{equation} \label{eq:simple_hyperbolic_heat_kernel_bound}
    \mheat{t}{x}{y}{\hyperbolic} \lesssim \frac{1}{t}\exp\left(-{\frac{{d_\hyperbolic(x,y)}^2}{8t}}\right) \, .
\end{equation}
We use this to bound the sum on the right hand side of \eqref{eq:partition_of_the_count} by
\begin{equation*}
    \sum_{m=1}^{\infty} \sum_{T \in \Gamma(m)} \mheat{t}{\tilde x}{T \tilde x}{\hyperbolic} \lesssim \frac{1}{t} \sum_{m=1}^\infty \exp\left( m - m^2 / 8t \right) \lesssim 1 \, ,
\end{equation*}
since the term over $m$ is convergent. For the sum over $T \in \Gamma(0)$ we note that $d_{\hyperbolic}(\tilde x,T\tilde x) \geq \inj{x} \gtrsim 1$ for every $T \in \Gamma(0)$ so by \eqref{eq:simple_hyperbolic_heat_kernel_bound} we obtain a contribution of another constant. 
\end{proof}

We next bound the behavior of $\mheat{t}{x}{x}{\S}$ in the thin part.

\begin{proposition}\label{prop:heat_kernel_trace_depends_on_local_geometry} 
    Let $\gamma$ be a simple closed geodesic of length $\ell \leq 2\sinh^{-1}(1)$. Then for any $t \in [1/2,1]$ we have     \begin{equation}\label{eq:heat_kernel_trace_depends_on_local_geometry}
        \int_{\truncated{C}(\gamma)} \mheat{t}{x}{x}{\S} ~dx \lesssim  \log(1/\ell) \, .
    \end{equation}
\end{proposition}

\begin{proof} We apply the classical Li and Yau theorem (see Corollary 3.1 in \cite{li_yau_main} with $\alpha=3/2$ and $\varepsilon=1/2$) which asserts that if $\S$ is a complete Riemannian manifold without boundary and with Ricci curvature bounded from below by $-K$, then 
\begin{equation} \label{eq:li_yau}
    \mheat{t}{x}{y}{\S} \lesssim \vol(B_\S(x,\sqrt{t}))^{-1/2}\vol(B_\S(y,\sqrt{t}))^{-1/2}\exp\left(C Kt-\frac{d_{\S}(x,y)^2}{4.5t} \right) \, ,
\end{equation}
where $C<\infty$ is a universal constant. Using the above bound yields 
\begin{equation}
    \mheat{t}{x}{x}{\S} \lesssim \frac{1}{\vol(B_\S(x,\sqrt{t}))} \, , \label{eq:beginning_of_li_yau}
\end{equation}
as long as $t\leq 1$. Let $x\in \truncated{C}(\gamma)$ have Fermi coordinates $(\rho, \theta)$ relative to $\gamma$, where $\rho \geq 0$ without loss of generality. The ball $B-\S(x,\sqrt{t})$ is contained in the collar $C(\gamma)$. Hence, if we let $\eta$ be the unit-speed geodesic given by $\eta(s) = \varphi_\gamma(\rho +s, \theta)$ with $s\in[0,\sqrt{t}/2]$ (that is, $\eta$  extends outwards from $x$, is perpendicular to $\gamma$ and has length $\sqrt{t}/2$), then $B_\S(x, \sqrt{t})$ contains the tubular neighborhood of $\eta$ given by $\varphi_\eta([0, \sqrt{t}/2] \times[0, r])$ where $r\asymp \inj{x}$ (which is possible since $\inj{x} \lesssim 1$). We thus obtain that $\vol(B_\S(x,\sqrt{t})) \gtrsim \inj{x}$. Plugging this volume estimate into the Li-Yau bound \eqref{eq:beginning_of_li_yau}, by \cref{prop:injradius_estimate} we thus obtain 
 \begin{align*}
    \int_{\truncated{C}(\gamma)} \mheat{t}{x}{x}{\S} ~dx  &\lesssim \int_0^{1} \int_{0}^{W(\ell)} \frac{1}{\inj{x}} \ell \cosh(\rho) ~d\rho d\theta \lesssim W(\ell) \lesssim \log(1/\ell) \, .
\end{align*}
\end{proof}

\subsection{Proof of \texorpdfstring{\cref{thm:main_heat_kernel}}{}}
It is well known that $p_t(x,x)\geq 1/\vol(S)$ so the expression in the absolute value is always non-negative. By the spectral theorem \cite[][Section VI.1, Sturm-Liouville decomposition]{chavel_book}, 
\begin{equation}\label{eq:convergence_to_uniform_basic}
    \int_{\S} \Big[ \mheat{t}{x}{x}{\S} - \frac{1}{\vol(\S)} \Big] ~dx = \sum_{k=1}^{\infty}e^{-t \lambda_k} =  \sum_{k=1}^{2g-3}e^{-t \lambda_k} + \sum_{k=2g-2}^{\infty}e^{-t \lambda_k} \, . 
\end{equation}
Applying \cref{thm:main_eigenvalue_lower_bound} and recalling that $g\asymp \vol(\S)$ we bound the first sum on the right-hand side by
\begin{equation}\label{eq:convergence_to_uniform_part_1}
    \sum_{k=1}^{2g-3}e^{-t \lambda_k} \leq \sum_{k=1}^{\infty}e^{-ct \frac{k^2}{\aw(\S)\vol(\S)^2}} \lesssim \int_0^{\infty} \exp \left(-ct \frac{x^2}{\aw(\S)\vol(\S)^2} \right)dx \asymp \vol(\S) \sqrt{\frac{\aw(\S)}{t}} \, .
\end{equation}    
For the second sum on the right-hand side, a result by Otal and Rosas \cite[][Theorem 1]{otal_rosas_eigenvalues_must_be_large} states that $\lambda_{2g-2} > \frac{1}{4}$. Denoting $t = \tau + 1/2$ for $\tau \geq 1/2$, we get
\begin{equation*}
    \sum_{k=2g-2}^\infty e^{-t \lambda_k} = \sum_{k=2g-2}^\infty e^{ -(\tau+1/2) \lambda_k} \leq e^{-\tau/4} \sum_{k=0}^{\infty}e^{-\frac{1}{2} \lambda_k} \lesssim e^{-t/4} \int_{\S} \mheat{1/2}{x}{x}{\S} ~dx \, .
\end{equation*}

Let $\gamma_1, \ldots, \gamma_s$ be the set of all simple closed geodesics of length $\leq 2\sinh^{-1}(1)$ in $\S$. We appeal to  \cref{prop:counting_argument_by_injradius} and \cref{prop:heat_kernel_trace_depends_on_local_geometry} with $t = 1/2$; since $\mheat{1/2}{0}{0}{\hyperbolic}\asymp 1$, we get
\begin{align*}
    \int_{\S} \mheat{1/2}{x}{x}{\S} ~dx &\lesssim \vol(\S) + \sum_{i=1}^{s}\log \left(\frac{1}{\ell(\gamma_i)} \right) \lesssim \vol(\S) \sqrt{\aw(\S)} \, ,    
\end{align*}
since $\sum_{i=1}^{s}\log(1/\ell(\gamma_i))\lesssim\sum_{i=1}^{s} \ell(\gamma_i)^{-1/2}$ which by Cauchy-Schwartz is at most $\sqrt{s\sum_{i=1}^{s} \ell(\gamma_i)^{-1}}$ and we get the above inequality using \cref{cor:geodesic_integral_correspondence}. All this gives that
\begin{equation*}
    \sum_{k=2g-2}^{\infty}e^{-t \lambda_k} \lesssim e^{-t/4} \vol(\S) \sqrt{\aw(\S)} \, ,
\end{equation*}
which together with \eqref{eq:convergence_to_uniform_basic} and \eqref{eq:convergence_to_uniform_part_1} concludes the proof.    \qed

\begin{remark}
The choice of $t=1$ in \cref{thm:main_heat_kernel} is arbitrary since it is the large values of $t$ that are interesting. The same proof shows that for any $t_0>0$ there exists $C(t_0)<\infty$ such that \eqref{eq:heatkernelBound} holds for all $t\geq t_0$ and $C=C(t_0)$.
\end{remark}

\section{Sharpness} \label{sec:sharpness}
Let $I\geq 1$ be real and $g \geq 2$ be an integer.  Set $\eps = \min \{\frac{1}{I}, \sinh^{-1}(1)\}$ and $n=g-1$. Let $Y$ be a pair of pants with boundary geodesics of lengths $(1,1,\eps)$, and let $X$ be the surface obtained by gluing 
two identical copies of $Y$ along their two boundary geodesics of length $1$, without twists. Note that $X$ has two boundary components, each of length $\eps$. Finally, let $\S$ be the surface of genus $g$ obtained by gluing together without twists $n$ copies of $X$ in a cycle. See \cref{fig:cycle_surface}.

\begin{figure}[!ht]
    \centering
    \includegraphics[width=0.3   \textwidth]{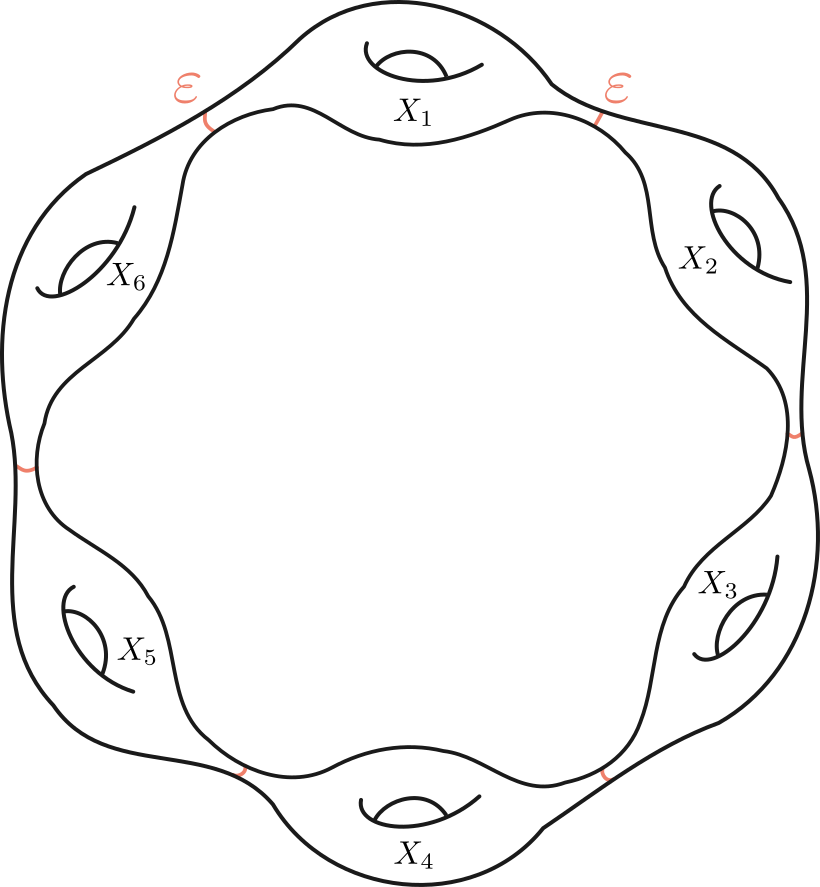}
    \caption{A schematic of $S$ for $n=6$.} 
    \label{fig:cycle_surface}
\end{figure}

By \cref{cor:geodesic_integral_correspondence}, $\aw(\S) \asymp 1/\eps$, and so $\aw(\S) \gtrsim I$. We claim that the eigenvalues of the Laplacian on $\S$ satisfy 
\begin{equation*}
  \lambda_k\leq \frac{C \eps k^2}{g^2} \, ,  
\end{equation*}
where $k\in\{1,\ldots, n \}$ and $C<\infty$ is a universal constant. We will do this by constructing appropriate test functions for the minimax principle.

Let $k \leq n$ be fixed. Consider  $k$ disjoint connected subsets $\S_1, \ldots, \S_k$ of $\S$, where each $\S_i$ is a concatenation of $\floor{\frac{n}{k}}$ consecutive copies of X. We assume for simplicity that $\floor{\frac{n}{k}}$ is odd and write $\floor{\frac{n}{k}} = 2m+1$; the even case is handled similarly. For each $i$ we define a function $f_i(x)$ whose support is $\S_i$, so that the functions $\{f_i\}_{i=1}^k$ have disjoint supports. Let us describe $f_1$; the other functions are constructed similarly. Assume that $S_1$ is the concatenation of $X_1,\ldots, X_{2m+1}$. If $x \in X_j$ for some $j\in\{1,\ldots, 2m+1\}$ and $\dist{x}{\partial X_j} \geq 1$, then we set $f_1(x) = \frac{j}{m+1}$ if $ j \leq m+1$ and $f_1(x) = 1 - \frac{j-(m+1)}{m+1}$ if $ j > m+1$. When $\dist{x}{\partial X_j} < 1$, we set the values of $f_i$ so that they interpolate linearly between the two values in the boundary $\{x : \dist{x}{\partial X_j} = 1\}$. 

It is straightforward to verify that $\norm{f_i}_2 \asymp \sqrt{m}$ and $\int_\S \abs{\grad f_i}^2 ~d\S \lesssim \eps/m $. The functions $\{f_i\}$ are not smooth, but one can easily make them smooth (say, by mollifying, as we did in \eqref{def:f}) and still have the same bounds on the norm and the Dirichlet energy as above. Thus, the minimax principle \cite[][Theorem 8.2.1 (i)]{buser_book} and the fact that $m \asymp n/k$ gives that $\lambda_{k}\lesssim \eps k^{2}/n^{2}$. This is the required bound for $k\in \{1,\ldots g-1\}$ and it implies the bound for $g\leq k\leq 2g-3$ by properly adjusting the constant $C$.

Lastly, this upper bound on $\lambda_k$ for $k\in\{1,\ldots,n\}$ immediately shows that the reverse inequality in \cref{thm:main_heat_kernel} holds. Indeed, we plug in this upper bound for the first $n$ eigenvalues in \eqref{eq:convergence_to_uniform_basic} and all other terms we bound below by $0$.

\section*{Acknowledgements} The authors are supported by ERC consolidator grant 101001124 (UniversalMap) as well as ISF grants 1294/19 and 898/23. We thank Rotem Assouline for his help with the proof of \cref{prop:change_of_metric}; Boaz Klartag, Ze'ev Rudnick and Perla Sousi for useful discussions; and Yuhao Xue for pointing out errors in an earlier version.

\printbibliography

\end{document}